\documentclass[11pt]{article}
\usepackage{amsmath, amsfonts, amsthm, amssymb}
\usepackage[]{pstricks,pst-node,pst-text,pst-3d}
\usepackage{graphics}

\newcommand{\T}{\top}

\DeclareMathOperator*{\esssup}{ess\, sup}

\def\del{\partial}

\newcommand{\RR}{\mathbb{R}}
\newcommand{\eps}{\varepsilon}

\newcommand{\D}{\mathrm{D}}

\newcommand{\PP}{\mathbb{P}}
\newcommand{\rank}{\mbox{rank}}

\def\clap#1{\hbox to 0pt{\hss#1\hss}}

\newcommand{\SP}{\hspace{1pt}}
\newcommand{\DSP}{\hspace{2pt}}

\newtheorem{theorem}{Theorem}

\newtheorem*{corollary*}{Corollary}
\newtheorem*{definition*}{Definition}
\newtheorem{definition}{Definition}
\newtheorem{lemma}{Lemma}
\newtheorem{remark}{Remark}
\newtheorem*{remark*}{Remark}
\newtheorem*{mtheorem*}{Main Theorem}

\definecolor{cstblue}{rgb}{0,0,0.85}
\definecolor{cstred}{rgb}{1,0,0}

\usepackage{fullpage}

\title{
{\large\bf RELATIVE ENTROPY IN HYPERBOLIC  RELAXATION FOR BALANCE LAWS}\\ {\footnotesize
(In: {\it Communications in Mathematical Sciences}, 12-6, pp. 1017-1043 (2014))}
}

\numberwithin{lemma}{section}
\numberwithin{theorem}{section}
\numberwithin{remark}{section}
\numberwithin{equation}{section}
\numberwithin{definition}{section}

\author{Alexey Miroshnikov\thanks {Department of Mathematics and Statistics, University of Massachusetts
Amherst, MA 01003, USA, amiroshn@math.umass.edu.} \;   and \, Konstantina Trivisa
\thanks{Department of Mathematics,
University of Maryland, College Park, MD 20742, USA, trivisa@math.umd.edu. }}

\date{}

\begin{document}

\numberwithin{lemma}{section}
\numberwithin{theorem}{section}
\numberwithin{remark}{section}
\numberwithin{equation}{section}
\numberwithin{definition}{section}

\maketitle

\begin{center}
\textbf{Abstract}
\end{center}
We present a  general framework for the approximation of systems of hyperbolic balance
laws. The novelty of the analysis lies in the construction of suitable relaxation systems
and the derivation of a delicate estimate on the relative entropy. We provide  a direct proof of
convergence  in the smooth regime for a wide class of physical systems. We present
results for systems arising in materials science, where the presence of source terms
presents a number of additional  challenges and  requires delicate treatment. Our
analysis is in the spirit of the framework introduced by Tzavaras  \cite{Tz05} for systems of
hyperbolic conservation laws. \setcounter{tocdepth}{1}

\section{Introduction}
We present a general framework for the approximation of systems of hyperbolic balance
laws,
\begin{equation}\label{BLAWINTRO}
\del_t u+\del_\alpha f_\alpha( u )=g(u) \quad u(x,t)\in\RR^n,\;x\in\RR^d,
\end{equation}
by relaxation systems presented in the form of the extended system
\begin{equation}\label{RLXBLAWSYS}
\del_tU+\del_\alpha F_\alpha(U)=\frac{1}{\eps}R(U)+G(U), \quad U(x,t)\in\RR^N,\;x\in\RR^d
\end{equation}
in the regime where the solution of the limiting  system (as $\eps \to 0$) is smooth.
Motivated by the structure of physical models and the analysis in \cite{CHENLEVLIU}, we
deal with relaxation systems of type \eqref{RLXBLAWSYS} which are equipped with a
globally defined, convex entropy $H(U)$ satisfying
\begin{equation}\label{RXENTIDINTRO}
\del_tH(U)+\del_\alpha Q_\alpha(U)= \frac{1}{\eps} \D H(U) R(U) + \D H(U) G(U).
\end{equation}

\par\smallskip

The problem of numerical approximation of nonlinear hyperbolic balance laws is extremely
challenging. In the present article  we identify a class of relaxation schemes suitable
for the approximation of solutions to certain systems of hyperbolic balance laws arising
in continuum physics. The relaxation schemes proposed in our work provide a very
effective mechanism for the approximation of the solutions of these systems with a very
high degree of accuracy.

\par\smallskip

The main contribution of the present article to the existing theory can be characterized
as follows:
\begin{itemize}
\item This work provides a general framework describing how, given a physical system
governed by a hyperbolic balance law, one can construct an extended system endowed with a
globally defined, convex entropy $H(U)$ and the resulting relaxation system for its
approximation. It has the potential of being of use for the construction of suitable
approximating schemes for a variety of hyperbolic balance laws. Our analysis treats a
large class of physical systems such as the system of elasticity \eqref{ELASTSYST} (c.f.
Section \ref{S6}), two phase flow models \eqref{ISCOMB} (c.f. Section \ref{S7}), and
general symmetric hyperbolic systems \eqref{SYMSYST} (c.f. Section \ref{S8}). In the
latter application the relaxation of the hyperbolic system is obtained by relaxing
$n$-vector flux components.

\item Our framework is applicable in the multidimensional setting and provides a rigorous
proof of the relaxation  limit and a rate of convergence for a large class of physically relevant
hyperbolic balance laws. As it is well known, results for multidimensional systems of
hyperbolic balance laws are limited in the literature. In addition, our analysis treats a
large class of source terms: those satisfying a special mechanism that induces
dissipation as well as more general source term.

\end{itemize}

\par\smallskip

We establish convergence of  weak solutions of \eqref{RLXBLAWSYS} to solutions of the
equilibrium system \eqref{BLAWINTRO} via a relative entropy argument which relies on
\eqref{RXENTIDINTRO}. The proof provides a rate of convergence. The relative entropy
method relies on the ``weak-strong" uniqueness principle established by Dafermos for
systems of conservation laws admitting convex entropy functional \cite{Dafermos79}, see
also DiPerna \cite{DiPerna79}. In addition to the pioneer work of Dafermos and DiPerna,
the relative entropy method has been successfully used to study hydrodynamic limits of
particle systems \cite{BV2009, GLT-2009, MV2008, Mourragui-1996,Yau-1991}, hydrodynamic limits
from kinetic equations to multidimensional macroscopic models \cite{Brenier2000, BV2005,
KMT-2012}, as well as the convergence of numerical schemes in the context of
three-dimensional polyconvex elasticity \cite{LT, MT11}.

\par\smallskip

The main ingredients of our approach can be formulated as follows:
\begin{itemize}
\item  A relative entropy inequality  which provides a simple and direct convergence
framework  before formation of shocks. The reader may contrast the present framework  to the
classic convergence framework for relaxation limits, which proceeds through analysis of
the linearized (collision or relaxation) operator \cite{Yong-2004}.

\item Physically grounded structural hypotheses imposed on the relaxation system. These
structural hypotheses  will be of use   for the  derivation of  the relative entropy
inequality and for the proof of the desired convergence. The relative entropy computation
hinges on {\it entropy consistency} \cite{Tz05}, that is, the restriction of the entropy
pair $H-Q_{\alpha}$ on the manifold of Maxwellians $${\mathcal M}:=\{U\in\RR^N: R(U)=0
\}=\{U\in\RR^N: U=M(u), \,\, u\in\RR^n \}$$ induces an entropy pair for the equilibrium
system \eqref{BLAWINTRO} in the form
\begin{equation*}
    \eta(u)=H(M(u)), \quad q_{\alpha}=Q_{\alpha}(M(u)).
\end{equation*}

\item A physically motivated dissipation mechanism (in the sense of \eqref{SRCDSP})
associated with the source term in \eqref{RLXBLAWSYS} with respect to the manifold of
Maxwellians on which relaxation takes place. The dissipation mechanism on
\eqref{RLXBLAWSYS} induces weak dissipation on the equilibrium balance law
\eqref{BLAWINTRO} due to {\em source consistency} requirement (c.f. Section
\ref{SRCCONST}). The concept of {\em weak dissipation} for hyperbolic balance laws was
introduced by Dafermos in \cite{Dafermos06}. To realize the role of dissipation in the
present context, the reader may contrast the result of Theorem \ref{WKSRCTHM}  for weakly
dissipative source terms with Theorem \ref{SRCLIPTHM} which  corresponds to the case of a
general source.
\end{itemize}

\par\smallskip

The paper is organized as follows: In Section \ref{S2} we present the structural
hypotheses on \eqref{RLXBLAWSYS} which are of use in the derivation of the relative
entropy inequality  and the proof of the desired convergence. In Section \ref{S3} we present
the main theorems of this article for two different classes of source terms. In Section
\ref{S4} we define the concept of {\em relative entropy} $H^r(U^{\eps}, M(\bar{u}))$ and
entropy fluxes $Q^r_{\alpha}(U^{\eps}, M(\bar{u}))$. Section \ref{S5} contains  the proof
of the main result, which is based on error estimates for the approximation of the
conserved quantities by the solution of the relaxation system. Applications to nonlinear
elasticity and two phase flow models (combustion) are presented in Section \ref{S6} and
Section \ref{S7}, respectively. Finally, Section \ref{S8} provides a general framework
describing how,  given a physical system governed by a symmetric hyperbolic balance law,
one can construct an extended system and the resulting relaxation system for its
approximation.

\section{Notation and Hypotheses}\label{S2}
For the convenience of the reader we collect in this section all the relevant notation and hypotheses.
Here and in what follows:
\begin{enumerate}
\item[] {\bf 1.} $G, R, F_{\alpha}, \alpha =1, \dots, d$   denote the mappings
$G, R,  F_{\alpha}: \RR^N \to \RR^N,$
 whereas $g, f$ denote the maps
 $g, f: \RR^n\to \RR^n.$
In our presentation, $G(U), R(U), F_{\alpha}(U), g(u), f(u)$ are treated as column vectors.

\item[] {\bf 2.} $\D$, $\D_u$  denote the differentials with respect to the state vectors $U\in \RR^N$ and $u\in\RR^n$ respectively. When used in conjunction with matrix notation, $\D$ and $\D_u$ represent  a row operation:
\begin{equation*}
\D=[\del/\del U^1, \dots, \del/\del U^N ], \quad  \D_u=(\del/\del u^1, \dots, \del/\del u^n ).
\end{equation*}
\item[] {\bf 3. }
The symbol $\del_{\alpha}$  denotes the derivative with respect to $x_{\alpha}$, $\alpha=1,\dots,d$.
The summation convention over repeated indices is employed throughout the article: repeated indices are summed over the range $1,\dots,d$.
\end{enumerate}

\par\smallskip

Motivated by theoretical studies  \cite{JK-2000, LT, Tz05} as well as computations
devoted to the approximation of the hyperbolic systems of conservation laws and  kinetic
equations \cite{ Caflish79} by relaxation schemes, our analysis is based on the following
assumptions:
\begin{itemize}
\item The manifold $\mathcal{M}$ of Maxwellians (the
equilibrium solutions $U_{eq}$ to the equation $R(U)=0$) can be
parameterized by $n$ conserved quantities
\begin{equation}\tag{H1}\label{MAXWPAR}
U_{eq}=M(u), \quad u\in\RR^n.
\end{equation}

\item $\nabla R(U)$ satisfies the nondegeneracy condition
\begin{equation} \tag{H2}\label{KERGRADR}
\begin{cases}
\mbox{dim} \SP \mathcal{N}(\nabla R(M(u)))&=n\\
\mbox{dim} \SP \mathcal{R}(\nabla R(M(u)))&=N-n
\end{cases}
\end{equation}

\item There exists a projection matrix
\begin{equation*}\label{PROJMTX}
\PP:\RR^N\!\rightarrow\RR^n \quad \mbox{with} \quad
{\mbox{rank}}\SP \PP = n
\end{equation*}
corresponding to Maxwellians that determines the conserved
quantity
$$u=\PP U\,\,\, \mbox{and satisfies}$$
\begin{equation}\tag{H3}\label{PROJPROP}
    \quad \PP M(u)=u \quad \mbox{and} \quad \PP R(U)=0 \quad \mbox{for all} \quad u \in \RR^n, U\in \RR^N.
\end{equation}
\end{itemize}
In this case, the corresponding system of balance laws for
conserved quantities is given by
\begin{equation}\label{BLAW1}
\del_t u+\del_\alpha\PP F_\alpha\bigl(M(u)\bigr)=\PP
G\bigl(M(u)\bigr)
\end{equation}
which can be rewritten in the form
\begin{equation}\nonumber
\del_t u+\del_\alpha f_\alpha( u )=g(u)
\end{equation}
with $f$, $g$ defined by
\begin{equation}\label{BLAWFLXSRC}
\begin{aligned}
f_\alpha(u)&:=\PP F_\alpha(M(u)),\quad g(u):=\PP
G\bigl(M(u)\bigr).
\end{aligned}
\end{equation}
The system of balance laws \eqref{BLAW1} is resulting by applying $\PP$ to
\eqref{RLXBLAWSYS}, letting $\eps\to 0$, and then using the fact that at the equilibrium
$U_{eq}=M(u)$, $u=\PP U_{eq}$.

\par\smallskip

Our analysis exploits the entropy structure of the relaxation systems under consideration. Below are stated the main structural assumptions on \eqref{RLXBLAWSYS}.


\subsection{Entropy Structure}

Some additional assumptions on the system \eqref{RLXBLAWSYS} read:
\begin{itemize}
\item The system \eqref{RLXBLAWSYS} is equipped with a globally
defined entropy $H(U)$ and corresponding fluxes $Q_{\alpha}(U)$,
$\alpha=1,\dots,d$, such that
\begin{equation}\tag{H4}\label{RXENTPROP1}
\begin{aligned}
& H:\RR^N \rightarrow \RR  \,\,\, \mbox{is convex}, \\
& {\D}H(U) \SP {\D}F_{\alpha}(U)={\D} Q_{\alpha}(U) .
\end{aligned}
\end{equation}
\item The entropy $H(U)$ is such that
\begin{equation}\tag{H5}\label{RXENTPROP2}
\begin{aligned}
D(U):= - \D H(U) R(U) & \geqslant 0, \quad U\in \RR^N.\\
\end{aligned}
\end{equation}
The entropy equation for the relaxation system \eqref{RLXBLAWSYS} in that case is given
by
\begin{equation}\label{RXENTID}
\del_tH(U)+\del_\alpha Q_\alpha(U) = - \frac{1}{\eps} \D(U) + \D H(U) G(U).
\end{equation}

\item {\it Entropy consistency.} The restriction of the entropy pair $H, Q_\alpha,$
\begin{equation}\tag{H6} \label{ENTCONSHYP}
\eta(u):=H\bigl(M(u)\bigr), \quad
q_\alpha(u):=Q_\alpha\bigl(M(u)\bigr),
\end{equation}
on the equilibrium manifold $\mathcal{M}$ is an entropy pair $\eta-q_{\alpha}$ for the
system of conserved quantities \eqref{BLAW1}, that is,
\begin{equation*}\label{BLAWENTPROP1}
{\D}_u \eta (u) {{\D}_u \SP f_{\alpha}}(u) = {\D}_u
q_{\alpha}(u),\quad u\in \RR^n.
\end{equation*}
In that case smooth solutions to \eqref{BLAW1} satisfy the additional balance law
\begin{equation}\label{BLAWENTID1}
\begin{aligned}
&\del_tH(M(u))+\del_\alpha Q_\alpha(M(u))={\D}_u \eta(u) g(u).\\
\end{aligned}
\end{equation}
\end{itemize}

In the sequel, we present some implications on the geometry of the manifold $\mathcal{M}$ obtained  as a consequence of the entropy structure of the relaxation systems. We  refer the reader to \cite{Tz05} for the details of the derivation in a relevant setting.

\subsection{Properties of $H$, $Q_{\alpha}$ on the manifold $\mathcal{M}$}

The geometric implications of the assumptions
\begin{equation*}
\begin{aligned}
&{\D}H(U)R(U) \leqslant 0, \quad R(M(u))=0\\
&\rank \SP \PP = n, \quad  \PP(M(u))=u, \,\quad \forall u\in\RR^n, \, U\in \RR^N\\
\end{aligned}
\end{equation*}
are the following \cite{Tz05}:
\begin{equation}\label{GEOMIMPL}
\begin{aligned}
\mathcal{R}({\D}R(M(u))) &= \mathcal{N}(\PP)\\
{\D}H(M(u)) {} \bigl[ {\D}R(M(u))A \bigr] &= 0, \quad \, \forall u\in\RR^n,\, A\in R^N\\
{\D}H(M(u)) {} V &= 0, \quad \, \forall V\in R^N \,\mbox{with} \, \PP V =0.\\
\end{aligned}
\end{equation}

\par\smallskip

Thus, the entropy consistency hypothesis \eqref{ENTCONSHYP} along with the property \eqref{GEOMIMPL}$_3$ imply that the
gradients of entropies $\eta$, $H$ are related by
\begin{equation}\label{ENTCONSIMPL}
{\D}_u \eta(u) \PP A \SP = \SP {\D}H(M(u)) A, \quad \forall A\in\RR^N.
\end{equation}
Then, in view of  \eqref{BLAWFLXSRC}$_2$, we have
\begin{equation*}
{\D}_u \eta(u) g(u) \SP = \SP {\D}H(M(u)) \SP G(M(u)), \quad \forall u \in \RR^n
\end{equation*}
and thus the entropy equation \eqref{BLAWENTID1} for conserved quantities may be written
as
\begin{equation}\label{BLAWENTID2}
\del_tH(M(u))+\del_\alpha Q_\alpha(M(u))={\D}H(M(u))\SP G(M(u)).
\end{equation}

\subsection{Dissipation}

Making use of the dissipation incorporated in the term $D(U)= -DH(U) R(U)$ we introduce
an additional  hypothesis, which plays the role of {\em relative dissipation}, a measure
of the distance between a relaxation state vector $U \in \RR^N$ and its ``equilibrium
version" $M(u) \in \RR^N$ with $u= \PP(U)$ on the manifold of Maxwellians $\mathcal{M}$.
More precisely,
\begin{itemize}
\item
We assume that  for
some $\nu>0$
\begin{equation}\tag{H7} \label{RLXDSP}
 -\bigl[{\D}H(U)-{\D}H(M(u))\bigr] \bigl[R(U)-R(M(u))\bigr] \SP \geqslant \SP \nu \SP |U-M(u)|^2
\end{equation}
for arbitrary $U \in\RR^N$ with $u=\PP U$.
\end{itemize}

Note that \eqref{RLXDSP} is stronger then the following assumption:
\begin{itemize}
\item For every ball $B_{r} \subset \RR^N$ there exists $\nu_r>0$ such that
\begin{equation}\tag{H7$^*$}\label{RLXDISSPHYPLOC}
-\bigl[{\D}H(U)-{\D}H(M(u))\bigr] \bigl[R(U)-R(M(u))\bigr] \geqslant
\nu_r\bigl|U-M(u)\bigr|^2
\end{equation}
for  $U,M(u)\in B_{r}$, where $u=\PP U$,
\end{itemize}
which will be of use in Theorem \ref{locbdd}.

\par\smallskip

Our analysis handles a large class of source terms. The following hypothesis will be
relevant to our subsequent discussion.
\begin{itemize}
\item  The source term $G(U)$ is {\it weakly dissipative with respect to the manifold
$\mathcal{M}$} in the sense of Definition \ref{D1.1}.
\begin{definition} \label{D1.1} \rm
We say that the source $G(U)$ is {\it weakly dissipative with respect to the manifold
$\mathcal{M}$}, if for all arbitrary $U, M(\bar{u}) \in \mathbb{R}^N$
\begin{equation}\tag{H8} \label{SRCDSP}
-\bigl[{\D}H(U)-{\D}H\bigl(M(\bar{u})\bigr)\bigr] \bigl[G(U)-G(M(\bar{u})) \bigr] \geqslant 0.
\end{equation}
\end{definition}
\end{itemize}
An alternative condition on the source $G(U)$, exploited in Theorem \ref{SRCLIPTHM},
reads:
\begin{itemize}
\item For every compact set $\mathcal{A} \subset \RR^N$ there exists $L_{\mathcal{A}}>0$
such that
\begin{equation}\tag{H9}\label{SRCLIP}
\bigl|G(U)-G(\bar{U})\bigr| \SP \leqslant \SP L_{\mathcal{A}} \SP |U-\bar{U}| \quad
\mbox{for all} \quad U\in \RR^N, \SP \bar{U}\in \mathcal{A}.
\end{equation}
\end{itemize}


\subsection{ Source consistency}\label{SRCCONST} \label{gweakdissip}\rm We first note
that the hypothesis \eqref{SRCDSP} is less restrictive than the requirement for $G$ to be
weakly dissipative, hence a special name for it: $\mathcal{M}$-weakly dissipativity.

\par\smallskip

We next point out that \eqref{SRCDSP} requires certain consistency between the source
terms $G(U)$ and $g(u)$ which are related by \eqref{BLAWFLXSRC}. Namely, take an
arbitrary $u\in\RR^n$ and set $U=M(u)$ in \eqref{SRCDSP}. Then, recalling
\eqref{ENTCONSIMPL} one concludes that \eqref{SRCDSP} implies that the source $g(u)$ in
the system \eqref{BLAW1} is {\it weakly dissipative}, that is
\begin{equation}\label{SRCCONSDSP}
-\Bigl({\D_u}\eta(u)-{\D_u}\eta (\bar{u}) \Bigr) \Bigl(g(u)-g(\bar{u})\Bigr) \SP
\geqslant \SP 0, \quad \, u,\bar{u}\in \RR^n.
\end{equation}
Thus, \eqref{SRCDSP} makes sense {\it only} when the source $g(u)$ in the equilibrium
system \eqref{BLAW1} is weakly dissipative in the sense of \eqref{SRCCONSDSP}.

\subsection{Weak solutions and entropy admissibility}

We introduce the notions of {\it weak solutions} and {\it entropy admissibility}
following the discussion in \cite[Sec. 4.3, 4.5]{Dafermos10}.
\begin{definition}\label{WKSOLDEF}
A locally bounded measurable function $U(x,t)$, defined on $\RR^{d} \times [0,T)$ and
taking values in an open set $\mathcal{O}\subset\RR^N$, is a weak solution to
\begin{equation}\label{CAUCHYPROB}
    \del_t U + \del_{\alpha} F_{\alpha} (U) = \frac{1}{\eps}R(U)+G(U)\SP, \quad U(x,0)  =
    U_0(x)\,,
\end{equation}
with $F,R,G: \mathcal{O} \to \RR^N$ Lipschitz, if
\begin{equation}\label{CAUCHYPROBWEAK}
\begin{aligned}
\int_{0}^T \int_{\RR^d} \,   \Big\{\del_t \widehat{\Phi} \, U  &+  \del_{\alpha}
\widehat{\Phi} \, F_{\alpha}(U) \Big\} \SP dx\SP dt + \int_{\RR^d} \widehat{\Phi}(x,0) \, U_0(x) \SP dx\\
&+ \int_0^T\int_{\RR^d} \widehat{\Phi}(x,t)
 \, \Big[ \frac{1}{\eps}R(U) +G(U) \Big] \SP dx\SP dt \, = \, 0
\end{aligned}
\end{equation}
for every Lipschitz test function $\widehat{\Phi}(x,t)$, with compact support in $\RR^d
\times [0,T)$ and values in $\mathbb{M}^{1\times N}$.
\end{definition}

\par
\begin{definition}\label{ADMDEF}
Assume that $H,Q_{\alpha}$ is an entropy-entropy flux pair of \eqref{CAUCHYPROB}. Then, a
weak solution $U(x,t)$ of \eqref{CAUCHYPROB}, in the sense of Definition \ref{WKSOLDEF},
defined in $\RR^d \times [0,T)$, is entropy admissible, relative to $H$, if
\begin{equation}\label{RLXADMWEAK}
\begin{aligned}
\int_{0}^T \int_{\RR^d} \,  \Big\{\del_t \varphi\,  & H(U)
+ \del_{\alpha} \varphi \, Q_{\alpha}(U)  \Big\} \SP dx\SP dt + \int_{\RR^d} \varphi(x,0) \, H(U_0(x)) \SP dx \\
& \quad + \int_0^T\int_{\RR^d} \varphi(x,t) \, \D H(U) \Big[\frac{1}{\eps}R(U)+G(U) \Big]
\SP dx\SP dt \, \geqslant \, 0
\end{aligned}
\end{equation}
for every nonnegative Lipschitz test function $\varphi(x,t)$, with compact support in
$\RR^d \times [0,T)$.
\end{definition}

\begin{remark}\rm
Note, a smooth solution $U^{\eps}$ of
\eqref{RLXBLAWSYS}  satisfies \eqref{RLXADMWEAK} identically as an equality and
therefore it is admissible.
 It is worth pointing out that  relaxation systems of type
\eqref{RLXBLAWSYS} are often designed to produce global smooth solutions. We refer the  reader to
\cite{HN03,Yong-2004} as well as \cite[Section 5.2]{Dafermos10} for further remarks.
A more detailed discussion about the existence of smooth solutions follows in the sequel.
\end{remark}

\section{Main Results}\label{S3}
In this section we present the main results of this article.
\subsection{$\mathcal{M}$-weakly dissipative source $G(U)$}
\begin{theorem}\label{WKSRCTHM}

Let $\SP\bar{u}(x,t)$ be a smooth solution of the equilibrium system
\eqref{BLAW1}, defined on $\RR^d \times [0,T]$, with initial data $\bar{u}_0(x)$. Let
$\{U^{\eps}(x,t)\}$ be a family of admissible weak solutions of the relaxation system
\eqref{RLXBLAWSYS} on $\RR^d \times [0,T)$, with initial data $U^{\eps}_0(x)$, and let
$u^{\eps}(x,t)=\PP U^{\eps}(x,t)$ denote the conserved quantity associated to
$U^{\eps}$.


\par\smallskip

Assume \eqref{MAXWPAR}-\eqref{SRCDSP} hold and suppose that:
\begin{itemize}

\item[(i)] $H(U)$, $F(U)$ in the relaxation system \eqref{RLXBLAWSYS} satisfy for some
$M,\SP\mu, \SP\mu' >0$
\begin{equation}\label{EXTSYSBOUNDSHYP}
\mu {\bf I} \SP \leqslant \D^2 H(U) \SP \leqslant \mu' {\bf I}, \quad |\D
F_{\alpha}(U)|<M, \:\quad
 U \in\RR^N.
\end{equation}

\item[(ii)] $\eta(u)=H(M(u))$, $f(u)=\PP F(M(u))$ satisfy for some $K>0$
\begin{equation}\label{EQSYSBOUNDSHYP}
\bigl|\D_u^3\eta(u)\bigr|\leqslant K, \: \quad  \bigl|{\D}^2_u f(u)\bigr|\leqslant K,
\:\quad u\in\RR^n.
\end{equation}
\end{itemize}
Then, for $R>0$ there exist constants $C=C(R,T,\nabla \bar{u},M,K)>0$ and $s>0$
independent of $\eps$ such that
\begin{equation}
\tag{ER} \label{ERREST} \int_{|x|<R}H^r(x,t) \SP dx \SP \leqslant \SP C
\biggl(\int_{|x|<R+st}H^r(x,0)\,dx+\eps\biggl), \quad \mbox{a.e.} \,\, t
\in [0,T).
\end{equation}
Moreover, if the initial data satisfy
\begin{equation}
\tag{CD}
\label{CONVDATAU} \int_{|x|<R+sT}H^r(x,0)\,dx\rightarrow
0\quad\text{as}\,\;\eps\downarrow 0,
\end{equation}
then
\begin{equation}
\tag{CS} \label{CONVU}
\esssup_{t\in[0,T)}\int_{|x|<R}|U^\eps-M(\bar{u})|^2(x,t)\,dx\rightarrow
0\quad\text{as}\;\,\eps\downarrow 0.
\end{equation}
\end{theorem}

\subsection{General source $G(U)$}

We now drop the assumption \eqref{SRCDSP} which leads to the following theorem.

\begin{theorem}\label{SRCLIPTHM}

Let $\bar{u}$ be a smooth solution of \eqref{BLAW1}, defined on $\RR^d
\times [0,T]$, with initial data $\bar{u}_0$, and $\{U^{\eps}\}$ a family of admissible
weak solutions of \eqref{RLXBLAWSYS} on $\RR^d \times [0,T)$, with initial data
$U^{\eps}_0$.


\par\smallskip

Assume \eqref{MAXWPAR}-\eqref{RLXDSP}, \eqref{SRCLIP} hold. \SP Suppose that $H(U)$,
$F(U)$, $\eta(u)$, and $f(u)$ satisfy $(i)$-$(ii)$ of Theorem \ref{WKSRCTHM}. Then, for
$R>0$ there holds the estimate \eqref{ERREST} for some constants $C=C(R,T,\nabla
\bar{u},M,K,L)>0$ and $s=s(M,\mu')>0$
independent of $\eps$. Moreover, if the initial data satisfy $\eqref{CONVDATAU}$, then
\eqref{CONVU} holds.
\end{theorem}


\subsection{Uniformly bounded $\bar{u},\,\{U^{\eps}\}$ }

If a priori bounds on the family of solutions $\{U^{\eps}\}$ are available, then it is
possible to weaken the requirements $(i)-(ii)$ of Theorems \ref{WKSRCTHM},
\ref{SRCLIPTHM}. For example, one may weaken the assumption for $H$ to be uniformly
convex and $\D F_{\alpha}$, $\D^2_u f_{\alpha}$, and $\D^3_u \eta$ to be uniformly
bounded.

\begin{theorem}\label{locbdd}

 Let $\bar{u}$ be a smooth solution of \eqref{BLAW1}, defined on $\RR^d
\times [0,T]$, with initial data $\bar{u}_0$, and $\{U^{\eps}\}$ a family of admissible
weak solutions of \eqref{RLXBLAWSYS} on $\RR^d \times [0,T)$, with initial data
$U^{\eps}_0$.


\par\medskip

Assume \eqref{MAXWPAR}-\eqref{ENTCONSHYP}, \eqref{RLXDISSPHYPLOC} hold. Suppose that:
\begin{itemize}
\item[(i)] $\{U^{\eps}\}$, $\{M(u^{\eps})\}$ and  $M(\bar{u})$ take
values in a ball $B_r \subset \RR^N$.

\item[(ii)] $H(U)\in C^2(\RR^N)$ is strictly convex. $F(U)$, $\eta(u)$, $f(u)$ are
smooth.


\item[(iii)] The source $G(U)$ either satisfies \eqref{SRCDSP} {or} is locally Lipschitz.

\end{itemize}\vspace{1pt}
Then, for $R>0$ there holds the estimate \eqref{ERREST} for some constant
\begin{equation*}\vspace{-2pt}
C=C\bigl( R,T,B_r,\SP \|\nabla \bar{u}\|_{W^{1,\infty}(\mathcal{C}_{(T,R)})}\bigr) \SP
> \SP 0
\end{equation*}\vspace{-2pt}
and the constant
\begin{equation*}
s ={\mu_r}^{-1}\sup_{U,V \in B_r} \sum_{\alpha}\bigl |{\D}^2H(U)\SP \D{F_\alpha}(V)\bigr
|\SP ,
\end{equation*}
where $\mathcal{C}_{(T,R)}$ denotes a cone
\begin{equation*}
\mathcal{C}_{(T,R)}=\bigl\{(x,t):\;0<t<T,\;|x|<R+s(T  -t)\bigr\}\SP,
\end{equation*}
and $\mu_r>0$ is a constant such that
\begin{equation*}
\mu_r {\bf{I}} \, < \, \D^2 H(U) , \quad U\in B_r.
\end{equation*}
Moreover, if the initial data satisfy $\eqref{CONVDATAU}$, then \eqref{CONVU} holds.
\end{theorem}

\section{Relative Entropy}\label{S4}

To compare the solution $U^{\eps}$ of the relaxation system \eqref{RLXBLAWSYS} and the solution $\bar{u}$ of the
equilibrium system \eqref{BLAW1}, we employ the notion of the relative entropy \cite{Dafermos79}. We define the
relative entropy and entropy-fluxes \cite{Tz05} among the two solutions by
\begin{equation}\label{RENTPAIR}
\begin{aligned}
\!\!\!\!H^r(U^{\eps},M(\bar{u}))& :=H(U^{\eps})-H (M(\bar{u}))-{\D}H (M(\bar{u})) \bigl[U^{\eps}-M(\bar{u})\bigr]\\[1pt]
\!\!\!\!Q^r_\alpha(U^{\eps}, M(\bar{u}))& := Q_\alpha(U^{\eps})
-Q_\alpha(M(\bar{u}))-{\D}H(M(\bar{u}))
\bigl[F_{\alpha}(U^{\eps})-F_{\alpha}(M(\bar{u}))\bigr].
\end{aligned}
\end{equation}
By \eqref{RXENTPROP2} we have
\begin{equation}\label{DSPTERMDEF}
D(U^{\eps})  = - {\D}H(U^{\eps}) R(U^{\eps}) \SP \geqslant \SP 0,
\end{equation}
that expresses the entropy dissipation of the relaxation system \eqref{RLXBLAWSYS}. In
view of \eqref{GEOMIMPL}$_3$ and the fact that $R(M(u))=0$ for all $u\in\RR^n$,
$D(U^{\eps})$ may be written in an alternative form
\begin{equation}\label{DSPTERMALT}
\begin{aligned}
D(U^{\eps})  = -\bigl[{\D}H(U^{\eps}) - {\D}H(M(u^{\eps}))\bigr] \bigl[R(U^{\eps})-R(M(u^{\eps}))\bigr] \SP \geqslant
\SP 0
\end{aligned}
\end{equation}
where $u^{\eps}=\PP U^{\eps}$. Finally, we denote by
\begin{equation}\label{SRCDISSP}
{S}(U^{\eps},M(\bar{u})):=-\bigl[{\D}H(U^{\eps})-{\D}H(M(\bar{u}))\bigr] \bigl[G(U^{\eps})-G(M(\bar{u}))\bigr]\\
\end{equation}
the term (not necessarily dissipative) associated with the source $G(U)$.

\par\smallskip

\par\smallskip

 Let $U \equiv U^{\eps}(x,t)$ be a smooth solution of the relaxation system
\eqref{RLXBLAWSYS}, $u(x,t)=\PP U(x,t)$ be the conserved quantity associated to $U$ and
$\bar{u}(x,t)$ be a smooth solution of the equilibrium system \eqref{BLAW1}. Then the
relative entropy $H^r(U,M(\bar{u}))$ satisfies \ \

\begin{lemma}[\bf Relative entropy identity ]\label{RENTINTYLMM}
Suppose $\bar{u}(x,t)$ is a smooth solution of the equilibrium system
\eqref{BLAW1}, defined on $\RR^d \times [0,T]$, with initial data $\bar{u}_0(x)$. Let $U
\equiv U^{\eps}(x,t)$ be any admissible weak solution of the relaxation system
\eqref{RLXBLAWSYS} on $\RR^d \times [0,T)$, with initial data $U_0(x)$, and let
$u(x,t)=\PP U(x,t)$ denote the conserved quantity associated to $U$. Then the relative
entropy $H^r(U,M(\bar{u}))$ satisfies
\begin{equation}\label{RELENTINTYWEAK}
\begin{aligned}
& \int_{0}^T \int_{\RR^d}   \Big\{ - \del_t \varphi  \, H^r(U,\bar{u}) - \del_{\alpha}
\varphi \, Q^r_{\alpha}(U,\bar{u}) \Big\} \SP dx\SP dt \\
&\qquad\qquad - \int_{\RR^d} \varphi(x,0) \, H^{r}(U_0, M(\bar{u}_0)) \SP dx \\
&\leqslant  \SP \int_{0}^T \int_{\RR^d} \varphi \, \Big\{- \frac{1}{\eps}D - S+J_1 + J_2 + J_3+J_4 \Big\} \SP dx\SP dt \\
\end{aligned}
\end{equation}
for every nonnegative Lipschitz test function $\varphi(x,t)$, with compact support in
$\RR^d \times [0,T)$, where
\begin{equation}\label{JDEFS}
\begin{aligned}
J_1 &:= - \bigl( {\D}^2_u\eta(\bar{u}) \del_\alpha \bar{u} \bigr)^{\T}
\bigl(f_\alpha(u)-f_\alpha(\bar{u})-
{\D_u}f_\alpha(\bar{u})(u-\bar{u})\bigr)\\[3pt]
J_2 &:= - \bigl({\D}^2_{u}\eta(\bar{u}) \del_\alpha \bar{u} \bigr)^{\T}  \PP\bigl[F_\alpha(U)-F_\alpha(M(u))\bigr]\\[3pt]
J_3 &:=  \SP g(\bar{u})^{\T} \Bigl({\D}_{u}{\eta}(u)^{\T}-{\D}_u\eta(\bar{u})^{\T}-{\D}_u^2\eta(\bar{u})^{\T}(u-\bar{u})\Bigr)\\[2pt]
J_4 &:=  \SP \bigl[{\D}H(U)-{\D}H(M(u))\bigr] \SP G(M(\bar{u}))\,.
\end{aligned}
\end{equation}
If,  in addition, $\{U^{\eps}\}$
are smooth solutions, then they identically satisfy \eqref{RLXADMWEAK}
as equality. As a consequence, the inequality
\eqref{RELENTINTYWEAK} for the relative entropy $H^r$ becomes the identity
\begin{equation}\label{RELENTID}
\del_t H^r + \del_{\alpha} Q^r_{\alpha} + \frac{1}{\eps}D + {S}=J_1 + J_2 + J_3 + J_4\,,
\quad (x,t)\in \RR^d \times [0,T).
\end{equation}

\end{lemma}

\begin{proof} Let us fix any nonnegative, Lipschitz continuous test function $\varphi(x,t)$,
compactly supported in $\RR^d \times [0,T)$. Since  $\bar{u}$ is smooth, from
\eqref{BLAWENTID2} it follows that $\eta(\bar{u})=H(M(\bar{u}))$ satisfies the entropy
identity
\begin{equation*}
\begin{aligned}
&\del_t H(M(\bar{u}))+\del_\alpha Q_\alpha(M(\bar{u}))={\D}H(M(\bar{u})) G(M(\bar{u}))\\
\end{aligned}
\end{equation*}
which in its the weak form reads
\begin{equation}\label{BLAWENTWEAK}
\begin{aligned}
&\int_{0}^T \int_{\RR^d} \, \Big(\del_t \varphi \, H(M(\bar{u})) +
\del_{\alpha} \varphi \, {Q}_{\alpha}(M(\bar{u}))  \Big)\SP dx\SP dt \\
&+ \int_{\RR^d} \varphi(x,0) \, H(M(\bar{u}_0))) \SP dx + \int_0^T\int_{\RR^d} \varphi \,
\D H(M(\bar{u})) G(M(\bar{u})) \SP dx\SP dt  = 0\,.
\end{aligned}
\end{equation}
Recall that $U$, an admissible weak solution of \eqref{RLXBLAWSYS}, with initial data
$U_0$, must satisfy the inequality \eqref{RLXADMWEAK}. Thus, upon subtracting
\eqref{BLAWENTWEAK} from \eqref{RLXADMWEAK}, we obtain
\begin{equation}\label{ENTIDDIFF}
\begin{aligned}
&\int_{0}^T \int_{\RR^d} \, \Big\{\del_t \varphi \, (H(U)-H(M(\bar{u}))) +
\del_{\alpha} \varphi \, (Q_{\alpha}(U)-Q_{\alpha}(M(\bar{u})))\Big\} \SP dx\SP dt\\
&+\int_0^T\int_{\RR^d} \varphi \Big\{ {\D}H(U) \Big[\frac{1}{\eps}R(U)+G(U)\Big]
-{\D}H(M(\bar{u})) G(M(\bar{u}))\Big\} dx \SP dt \\
&+ \int_{\RR^d} \varphi(x,0) \, (H(U_0)-H(M(\bar{u}_0))) \SP dx \,\geqslant \, 0\,.
\end{aligned}
\end{equation}


\par\smallskip

Next, recalling that $\bar{u}$ is a smooth solution of
\begin{equation}\label{BLAW2}
\del_t \bar{u}+\del_\alpha\PP F_\alpha\bigl(M(\bar{u})\bigr)=\PP G\big(M(\bar{u})\big)
\end{equation}
and that $\PP M(\bar{u})=\bar{u}$ we obtain the identity
\begin{equation}\label{BLAWWEAK}
\begin{aligned}
&\int_{0}^T \int_{\RR^d} \, \Big\{\del_t \Phi \, \PP M(\bar{u})  + \del_{\alpha}\Phi
\, \PP F_{\alpha}(M(\bar{u})) \Big\} \SP dx\SP dt \\
&\quad + \int_{\RR^d} \Phi(x,0) \, \PP M(\bar{u}_0(x)) \SP dx + \int_0^T\int_{\RR^d}
\Phi(x,t)
 \, \PP G(M(\bar{u})) \SP dx\SP dt = 0
\end{aligned}
\end{equation}
where $\Phi(x,t)$ is a Lipschitz continuous vector field with compact support in $\RR^d
\times [0,T)$ and values in $\mathbb{M}^{1\times n}$. Also, since $U$ is a weak solution
of \eqref{RLXBLAWSYS}, it must satisfy \eqref{CAUCHYPROBWEAK} which, with $\widehat{\Phi}
= \Phi \, \PP \in \mathbb{M}^{1\times N}$, reads
\begin{equation}\label{PROJRLXSYS}
\begin{aligned}
&\int_{0}^T \int_{\RR^d} \, \Big\{\del_t \Phi \, \PP  U  + \del_{\alpha}
\Phi \, \PP F_{\alpha}(U) \Big\} \SP dx\SP dt \\
&\qquad\quad + \int_{\RR^d} \Phi(x,0) \, \PP U_0(x) \SP dx + \int_0^T\int_{\RR^d}
\Phi(x,t)
 \, \PP G(U) \SP dx\SP dt = 0
\end{aligned}
\end{equation}
in view of the property $\PP R(U)=0$.


\par\smallskip

Now, we subtract \eqref{PROJRLXSYS} from \eqref{BLAWWEAK}, set the Lipschitz continuous
vector field $\Phi = \varphi \,\D_u \eta(\bar{u})$,
and recall the geometric relation \eqref{ENTCONSIMPL}, to get
\begin{equation}\label{PRODENTDIFF}
\begin{aligned}
&\int_{0}^T \int_{\RR^d} \Big\{ \del_t \varphi \, \D H(M(\bar{u})) \big[U - M(\bar{u})\big] \\
&\qquad\qquad  + \del_{\alpha} \varphi \, \D H(M(\bar{u}))
\big[F_{\alpha}(U) - F_{\alpha}(M(\bar{u}))\big] \Big\} \SP dx\SP dt \\
&+\int_{0}^T \int_{\RR^d} \varphi \, \Big\{ \bigl({\D}^2_u \eta(\bar{u}) \SP \del_t
\bar{u} \bigr)^{\T} \PP\bigl[U-M(\bar{u})\bigr] \\
&\qquad\qquad\qquad +  \bigl({\D}^2_u \eta(\bar{u}) \SP \del_{\alpha} \bar{u}
\bigr)^{\T} \PP\bigl[F_\alpha(U)-F_\alpha(M(\bar{u}))\bigr]\\
&\qquad\qquad\qquad  +\D H (M(\bar{u})) \big[ G(U) - G(M(\bar{u})\big]\Big\} \SP dx\SP dt \\
&+ \int_{\RR^d} \varphi(x,0) \,\D H(M(\bar{u}_0)) \big[U_0(x)-M(\bar{u}_0)\big] \SP dx \,
= \, 0
\end{aligned}
\end{equation}
\par\smallskip


The existence of an entropy pair $\eta-q_{\alpha}$ is equivalent
to the property
\begin{equation*}
\D_u^2\eta(v)\D_u f_\alpha(v)=\D_u f_\alpha(v)^{\T} \D_u^2 \eta(v), \quad \forall v \in
\RR^n
\end{equation*}
and therefore, in view of \eqref{BLAW2}, we have
\begin{equation*}
\begin{aligned}
    {\D}^2_u \eta(\bar{u}) \SP \del_t\bar{u} &= {\D}^2_u
    \eta(\bar{u})\Bigl( -\D_u f_{\alpha}(\bar{u}) \SP \del_{\alpha} \bar{u} +
    g(\bar{u})\Bigr)\\
    &= -\D_u f_{\alpha}(\bar{u})^{\T} {\D}^2_u\eta(\bar{u})\SP \del_{\alpha} \bar{u} +
    {\D}^2_u \eta(\bar{u}) g(\bar{u}).
\end{aligned}
\end{equation*}
Hence we must have
\begin{equation}\label{TWOTERMS}
\begin{aligned}
&         \bigl({\D}^2_u \eta(\bar{u}) \SP \del_t \bar{u}\bigr)^{\T} \PP\bigl[U-M(\bar{u})\bigr]
+\bigl({\D}^2_u \eta(\bar{u})\SP \del_{\alpha} \bar{u} \bigr)^{\T} \PP\bigl[F_\alpha(U)-F_\alpha(M(\bar{u}))\bigr]\\[1pt]
&  \quad =\bigl({\D}^2_u \eta(\bar{u})
\SP\del_{\alpha}\bar{u}\bigr)^{\T}\Bigl(f_\alpha(u)-f_\alpha(\bar{u})-\D_u
f_\alpha(\bar{u})(u-\bar{u})\Bigr)\\[1pt]
&\qquad  +  \bigl({\D}^2_u \eta(\bar{u}) \SP\del_{\alpha}\bar{u}
\bigr)^{\T}\PP\bigl[F_\alpha(U)-F_\alpha(M(u))\bigr]+\bigl(\D_u^2\SP\eta(\bar{u})g(\bar{u})\bigr)^{\T}(u-\bar{u})
\end{aligned}
\end{equation}
where we used \eqref{BLAWFLXSRC}, the fact that $u=\PP U$, and
$\PP M(\bar{u})=\bar{u}$.

\par\smallskip
Combining  \eqref{ENTIDDIFF}  with \eqref{PRODENTDIFF} - \eqref{TWOTERMS}
and recalling \eqref{DSPTERMDEF}, \eqref{DSPTERMALT} we obtain
%
\begin{equation}\label{RELENTIDINT}
\begin{aligned}
&\int_{0}^T \int_{\RR^d} \Big\{ \del_t \varphi \, H^r(U,\bar{u}) + \del_{\alpha} \varphi
\, Q^r_{\alpha}(U,\bar{u}) \Big\} \SP dx\SP dt \\
&+\int_{0}^T \int_{\RR^d} \varphi \, \Big\{ -\frac{1}{\eps}D - {S}+J_1 + J_2 + J_4
- \bigl({\D}^2_u\eta(\bar{u})g(\bar{u})\bigr)^{\T}(u-\bar{u}) \\
&\,\,\qquad\qquad\qquad + \bigl[{\D}H(M(u))-{\D}H(M(\bar{u}))\bigr] G(M(\bar{u})) \SP \Big\} \SP dx\SP dt \\
&+ \int_{\RR^d} \varphi(x,0) \, H^{r}(U_0, M(\bar{u}_0))\SP dx \, \geqslant  \, 0 \,.
\end{aligned}
\end{equation}
Observe that, in view of \eqref{BLAWFLXSRC}, \eqref{ENTCONSIMPL}, we have
\begin{equation*}
\begin{aligned}
\bigl[{\D}H(M(u))-{\D}H(M(\bar{u}))\bigr] G(M(\bar{u})) &=
\bigl({\D}_u \eta(u)-{\D}_u \eta(\bar{u})\bigr) g(\bar{u})
\end{aligned}
\end{equation*}
and hence
\begin{equation}\label{J3TERM}
\begin{aligned}
&\bigl[{\D}H(M(u))-{\D}H(M(\bar{u}))\bigr]G(M(\bar{u}))-\bigl({\D}^2_u\eta(\bar{u})g(\bar{u})\bigr)^{\T}(u-\bar{u})\\
&\quad\qquad\qquad  =   g(\bar{u})^{\T} \Bigl(D_u\eta(u)^{\T}-\D_u\eta(\bar{u})^{\T}-{\D}^2_u\eta(\bar{u})^{\T}(u-\bar{u})\Bigr)=J_3.\\
\end{aligned}
\end{equation}
Then from \eqref{RELENTIDINT}, \eqref{J3TERM} we get the desired inequality
\eqref{RELENTINTYWEAK}.
\end{proof}

\par
\section{Proof of Theorems via  Error Estimates}\label{S5}

To investigate the convergence of solutions $\{U^{\eps}\}$ of the relaxation system
\eqref{RLXBLAWSYS} to $M(\bar{u})$ in the smooth regime, one employs the
inequality \eqref{RELENTINTYWEAK} derived in the previous section. The preliminary
analysis of the  inequality indicates that the evolution of
$H^r(\cdot,t)$ depends heavily on the properties of the entropy $H(U)$, flux $F(U)$,
dissipative source $R(U)$ and, especially, the source $G(U)$.

\par\smallskip

\subsection{Proof of Theorem \ref{WKSRCTHM}}\label{S5.1}
\begin{proof}
The argument follows along the lines of [10, Theorem 5.2.1]. Fix
$\eps>0$. Since $U^{\eps}$ is an admissible weak solution of \eqref{RLXBLAWSYS} it must
satisfy \eqref{RLXADMWEAK}. Then \cite[Lemma 1.3.3]{Dafermos10} implies that the map $t
\to H(U^{\eps}(\cdot,t))$ is continuous on $[0,T)\backslash \mathcal{F}$ in
$L^{\infty}(\mathcal{A})$ weak$^*$, for any compact subset $\mathcal{A}\subset \RR^d$,
where $\mathcal{F}$ is at most countable.

\par\smallskip

We now fix $R>0$ and any point $t\in[0,T)$ of $L^{\infty}\,$weak$^*$
continuity of $H(U^{\eps}(\cdot,t))$ and let $\mathcal{C}_{(t,R)}$ denote the cone
\begin{equation*}
\mathcal{C}_{(t,R)}=\Bigl\{(x,\tau):\;0<\tau<t,\;|x|<R+s(t-\tau)\Bigr\}
\end{equation*}
where $s$ is a constant selected later. To prove the statement of the theorem we need to
monitor the evolution of the quantity
\begin{equation*}
\Psi(\tau)=\Psi(\tau \SP; \SP t, R):=\int_{|x|<R+s(t-\tau)}H^r(x,\tau)\,dx,\quad 0\leqslant\tau\leqslant t.
\end{equation*}

\par\smallskip

Clearly $U^{\eps}$, $\bar{u}$ satisfy the assumptions of Lemma \ref{RENTINTYLMM}
and hence there holds the relative entropy inequality
\begin{equation}\label{RELENTIDWKFORM}
\begin{aligned}
&\int_0^T\int_{\RR^d} \Bigl\{-H^r(x,\tau)\,\del_t\varphi - Q_\alpha^r(x,\tau)\,
\del_\alpha\varphi+\frac{1}{\eps}\varphi D \Bigr\} \SP dx \SP d\tau \\
& -\int_{\RR^d} H^r(x,0)\SP \varphi(x,0) \SP dx \SP \leqslant \SP
           \int_0^T\int_{\RR^d}\varphi \, \Big\{-S+J_{1}+J_{2}+J_{3}+J_{4}\Big\} \SP dx \SP d\tau
\end{aligned}
\end{equation}
where $D$, ${S}$, $J_k$, $k=1,\dots,4$ defined by \eqref{SRCDISSP} and $\varphi$ is
nonnegative Lipschitz continuous function compactly supported in $\RR^d\times[0,T)$.

\par\smallskip

Since the family $\{U^{\eps}\}$ together with $\bar{u}$ are not necessarily
uniformly bounded, to handle the flux term $Q^r$ 
we need to exploit the uniform convexity of the entropy $H(U)$.
From \eqref{RENTPAIR} and the assumption $(i)$ it follows that
there exists $c_1>0$ independent of $\eps$ such that
\begin{equation}\label{RELENTBOUND}
H^r\bigl(U^\eps,M(\bar{u})\bigr)\,\geqslant\,c_1\bigl|U^\eps-M(\bar{u})\bigr|^2.
\end{equation}

\par\smallskip

Now, by $\eqref{RENTPAIR}_2$ the relative entropy flux $Q_\alpha^r$ maybe written as
\begin{equation}\label{RELFLUXID1}
\begin{aligned}
Q_\alpha^r(U^{\eps},M(\bar{u}))&=\int_0^1 \SP {\D} Q_\alpha(\hat{U}(\beta)) \bigl[U^{\eps}-M(\bar{u})\bigr]\, d\beta\\
&\qquad - \int_0^1 {\D}H(M(\bar{u})) \Bigl[\D F_\alpha (\hat{U}(\beta))\bigl[U^{\eps}-M(\bar{u})\bigr]\Bigr] \, d\beta
\end{aligned}
\end{equation}
where $\widehat{U}(\beta):=\beta U^{\eps}+(1-\beta)M(\bar{u})$. Recalling \eqref{RXENTPROP1} we have
\begin{equation*}
\D{Q_\alpha}(\widehat{U}) \bigl[U^{\eps}-M(\bar{u})\bigr]\SP = \SP {\D}H(\widehat{U}) \SP \D F_\alpha
(\widehat{U})\bigl[U^{\eps}-M(\bar{u}) \bigr]
\end{equation*}
and hence \eqref{RELFLUXID1} becomes
\begin{equation}\label{RELFLUXID2}
\begin{aligned}
Q_\alpha^r &= \int_0^1\Bigl[{\D}H(\hat{U}(\beta))-{\D}H(M(\bar{u}))\Bigr] \Bigl[\D{F_\alpha}(\widehat{U}(\beta))\bigl[U^{\eps}-M(\bar{u})\bigr]\Bigr] \SP d\beta\\
&           = \bigl[U^{\eps}-M(\bar{u})\bigr]^{\T} \biggr(\int_0^1\int_0^1 \SP \beta \SP {\D}^2H(\widetilde{U})\SP
\D{F_\alpha}(\widehat{U}) \DSP d\gamma \SP d\beta \biggl) \bigl[U^{\eps}-M(\bar{u})\bigr]
\end{aligned}
\end{equation}
where
${\widetilde{U}}(\beta,\gamma):=\beta\gamma\SP{U^{\eps}}+(1-\beta\gamma)M(\bar{u})$.
Then, from \eqref{RELFLUXID2} and $(i)$ we conclude that
\begin{equation}\label{RELFLUXBOUND}
\sum_{\alpha}|Q_\alpha^r|\,\leqslant\,c_2|U^{\eps}-M(\bar{u})|^2
\end{equation}
for some $c_2=c_2(M,\mu')>0$ independent of $\eps$. Hence, in view of \eqref{RELENTBOUND}
and \eqref{RELFLUXBOUND}, we can choose $s>0$ such that
\begin{equation}\label{RELENTFLUXBOUND}
sH^r(x,\tau)+\sum_\alpha\frac{x_\alpha}{|x|}Q_\alpha^r(x,\tau)>0, \quad (x,\tau) \in
\RR^d \times[0,T{\color{cstblue})}.
\end{equation}

\par\smallskip

Next, take $\delta>0$ such that $t+\delta<T$ and select the test function
$\varphi=\varphi(x,\tau)$ as follows (cf. \cite[Theorem 5.3.1]{Dafermos10})
$$
\varphi(x,\tau)=\theta(\tau)\SP\gamma(x,\tau)
$$ where
\begin{equation*}
\begin{aligned}
\theta(\tau)&= \left\{
\begin{aligned}
&1,\qquad\qquad\qquad 0\leqslant\tau<t\\
&1-\frac{1}{\delta}(\tau-t),\quad t\leqslant\tau\leqslant t+\delta\\
&0,\qquad\qquad\qquad t+\delta\leqslant\tau,
\end{aligned}
\right.\\[3pt]
\gamma(x,\tau)&= \left\{
\begin{aligned}
&1,& &\tau>0,\,|x|-R-s(t-\tau)<0\\
&1-\frac{1}{\delta}\bigl(|x|-R-s(t-\tau)\bigr),& &\tau>0,\,0<|x|-s(t-\tau)-R<\delta\\
&0,& &\tau>0, \, \delta<|x|-R-s(t-\tau)
\end{aligned}
\right.
\end{aligned}
\end{equation*}
and use it in \eqref{RELENTIDWKFORM}. This gives
\begin{equation}\label{RELENTEST1}\vspace{3pt}
\begin{aligned}
&\frac{1}{\delta} \int_t^{t+\delta} \!\!\! \int_{|x|<R}H^r(x,\tau)\,dx\SP d\tau - \int_{|x|<R+st}H^r(x,0)\,dx \\[3pt]
&\qquad +  \frac{1}{\delta}\int_0^t\int_{0<|x|-R-s(t-\tau)<\delta} \Bigl(sH^r(x,\tau)+\sum_\alpha\frac{x_\alpha}{|x|}Q_\alpha^r(x,\tau)\Bigr)\,dx\SP
d\tau\\[3pt]
&\qquad +  \frac{1}{\eps}\int_0^t\int_{|x|<R+s(t-\tau)}\!\!\!D\,dx\SP d\tau+O(\delta)\\[3pt]
&       = \SP \int_0^t\int_{|x|<R+s(t-\tau)}\!\!\!\bigl(-{S}+J_{1}+J_{2}+J_{3}+J_{4}\bigr) \SP dx\SP d\tau.\\
\end{aligned}
\end{equation}
We next let $\delta\rightarrow 0^+$ in \eqref{RELENTEST1}. The second
integrals in \eqref{RELENTEST1} is nonnegative in view of \eqref{RELENTFLUXBOUND}.
Recalling \eqref{RLXDSP} and using the fact that
$H^r(U^{\eps}(\cdot,\tau),\bar{u}(\cdot,\tau))$ is weak$^*$ continuous in $L^{\infty}$ at
$\tau=t$ we conclude
\begin{equation}\label{RELENTEST2}
\begin{aligned}
&\int_{|x|<R}H^r(x,t)\SP dx+\frac{\nu}{\eps}\int\!\int_{\mathcal{C}_{(t,R)}} \! |U^{\eps}-M(u^{\eps})|^2 \SP dx\SP
d\tau + \int\!\int_{{\mathcal{C}_{(t,R)}}} \!{S} \, \SP dx\SP
d\tau\\[3pt]
&\quad \leqslant \, \int_{|x|<R+st}\!\!\! H^r(x,0) \SP dx
+\int\!\int_{{\mathcal{C}_{(t,R)}}}\!\Bigl(|J_{1}|+|J_{2}|+|J_{3}|+|J_{4}|\Bigr)\SP dx\SP
d\tau. \\
\end{aligned}
\end{equation}

\par\smallskip

We next estimate the terms on the right-hand side of \eqref{RELENTEST2}. Recalling
\eqref{JDEFS} and using $(i)$, $(ii)$, and the Young's inequality we obtain
\begin{equation}\label{J1J2J3J4EST}
\begin{aligned}
\int\int_{\mathcal{C}_{(t,R)}} \SP |J_{1}|+|J_{3}| \, dx\SP d\tau \SP
 \SP &\leqslant \SP C\int\!\int_{\mathcal{C}_{(t,R)}}\bigl|U^{\eps}-M(\bar{u})\bigr|^2 \SP dx\SP d\tau\\
\int\int_{\mathcal{C}_{(t,R)}} \SP |J_{2}|+|J_{4}| \, dx \SP d\tau \SP & \SP \leqslant
\SP\frac{\nu}{\eps}\int\!\int_{\mathcal{C}_{(t,R)}}\bigl|U^{\eps}-M(u)\bigr|^2\,dx+C\eps
\SP,
\end{aligned}
\end{equation}
where the constant $C=C(t,R,u,\nabla\bar{u},M,K)>0$ depends on the norms
\begin{equation}
\|\bar{u}\|_{W^{1,\infty}(\mathcal{C}_{(t,R)})}, \quad \|\bar{u}\|_{W^{1,2}(\mathcal{C}_{(t,R)})},
\end{equation}
and constants $M,K$ are introduced in $(i),(ii)$. Finally, by \eqref{SRCDSP}
\begin{equation}\label{SRCDISSPINTY}
\int\!\int_{\mathcal{C}_{(t,R)}} {S}(U^{\eps},M(\bar{u})) \, dx\SP d\tau \SP \geqslant
\SP 0.
\end{equation}
Then, combining \eqref{RELENTEST2}-\eqref{SRCDISSPINTY} and recalling $(i)_1$ we conclude
\begin{equation*}
\Psi(t \SP; \SP t, R)\,\leqslant\,\Psi(0 \SP; \SP t, R)+C\Bigl(\eps+\int_0^t\Psi(\tau \SP ; \SP t, R) \, d\tau\Bigr).
\end{equation*}
Since $R>0$ and $t\in[0,T]$ in the above inequality are arbitrary, we conclude via the Gronwall lemma.
\end{proof}

\begin{remark}\label{NOGROWTHINFO}\rm
The terms $J_1$, $J_3$ (in the proof of Theorem \ref{WKSRCTHM}) are bounded by $C
H^r(U^{\eps},M(\bar{u}))$, in view of \eqref{EXTSYSBOUNDSHYP}$_1$,
\eqref{J1J2J3J4EST}$_1$. This is one of the key features of the calculations that
eventually leads to the use of the Gronwall lemma.

\par\smallskip

The term $S(U^{\eps},M(\bar{u}))$ has a ``quadratic" structure similar to that $J_1,J_3$
and thus, one may think that there is no need in requiring \eqref{SRCDSP}. To this end,
we point out that if \eqref{SRCDSP} does not hold, then one has to make sure that
\begin{equation}\label{SASMP}
\int\!\int_{\mathcal{C}_{(t,R)}} S(U^{\eps},M(\bar{u})) \, dx \SP \leqslant \SP c\int\!\int_{\mathcal{C}_{(t,R)}} H^r(U^{\eps},M(\bar{u})) \, dx\\
\end{equation}
with $c=c(t,R)>0$ independent of $\eps$ (in order to exploit Gronwall lemma), and this is
not true in general. In this case, to ensure \eqref{SASMP}, one has to impose certain
regularity conditions on the source function $G(U)$.


\end{remark}

\subsection{Proof of Theorem \ref{SRCLIPTHM}}

In this section we drop the assumption \eqref{SRCDSP} and following Remark \ref{NOGROWTHINFO} require the source $G(U)$
to satisfy \eqref{SRCLIP}. This will ensure \eqref{SASMP} and thus following the analysis in the proof of Theorem \ref{WKSRCTHM} we obtain the result.

\subsection{Proof of Theorem \ref{locbdd}}
In the previous two sections we established convergence of weak solutions of the relaxation system \eqref{RLXBLAWSYS}
to the equilibrium system via the error estimate on the cone $\mathcal{C}_{(R,t)}$. Observe, however, that the bounds
imposed on $\D^2 H$, $\D^3_u \eta$ and $\D_u f_{\alpha}$, $\D F_{\alpha}$ in Theorem \ref{WKSRCTHM} are global. In
particular, the requirement that $H$ is uniformly convex on $\RR^N$ (which is used to handle the flux $Q^r_{\alpha}$ on
the boundary of the cone, see \eqref{RELENTFLUXBOUND}) is a very stringent condition that narrows significantly the
class of systems to which our error analysis may be applied.

\par\smallskip

Let us note at this point that if a priori (local) bounds on the family of solutions
$\{U^{\eps}\}$ are available, then it is possible to weaken the requirements $(i)-(ii)$
of Theorems \ref{WKSRCTHM}, \ref{SRCLIPTHM}. For example, one may weaken the assumption
for $H$ to be uniformly convex and $\D F_{\alpha}$, $\D^2_u f_{\alpha}$, and $\D^3_u
\eta$ to be uniformly bounded. This is indeed the case and the proof of Theorem
\ref{locbdd} follows using the line of argument presented in the proof of Theorem
\ref{WKSRCTHM}.

\par\smallskip

\section{Application to Elasticity}\label{S6}

Consider the relaxation of the (isothermal/isentropic) elasticity system:
\begin{equation}\label{ELASTSYST}
\begin{aligned}
\left(\begin{aligned}
&u\\
&v\\
\end{aligned}\right)_t -
\left(\begin{aligned}
&v\\
\sigma&(u)\\
\end{aligned}\right)_x = g(u,v)=\left(\begin{aligned}
&0\\
g_2(&u,v)\\
\end{aligned}\right)
\end{aligned}
\end{equation}
with the stress $\sigma(u)$ such that
\begin{equation}\label{GSTRESS}
\sigma(0)=0 \quad \mbox{and} \quad  0 < \gamma < \sigma'(u) < \Gamma\quad \mbox{for all} \quad u \in \RR^n.
\end{equation}
We assume that the source $g(u,v)$ satisfies one of the following:
\begin{itemize}
\item [$(i)$] {\it Either} $g$ is independent of $u$, that is $g(u,v)=g(v)$, and
satisfies
\begin{equation}\label{WKDSPELAST}
    \bigl(g_2(v)-g_2(\bar{v})\bigr)\bigl(v-\bar{v}\bigr) \leqslant 0, \quad \forall \SP
    v,\bar{v}\in \RR,
\end{equation}
\end{itemize}
\begin{itemize}
\item [$(ii)$] {\it or} \SP for every compact set $\mathcal{A} \subset \RR^2$ there
exists $L_{\mathcal{A}}>0$ such that
\begin{equation}\label{SRCLIPELAST}
|g_2(u,v)-g_2(\bar{u},\bar{v})| \leqslant L_{\mathcal{A}}
\bigl(|u-\bar{u}|+|v-\bar{v}|\bigr)
\end{equation}
for all $\SP (u,v)\in \RR^2$, $(\bar{u},\bar{v})\in\mathcal{A}$.
\end{itemize}
The system \eqref{ELASTSYST} is equipped with the entropy - entropy flux pair
$\bar{\eta},\bar{q}$ given by
\begin{equation}\label{ELASTENTPAIR}
\bar{\eta}(u,v) = \tfrac{1}{2}v^2 + \Sigma(u), \quad \bar{q}(u,v) = -\sigma(u)v \quad
\mbox{with} \quad \Sigma(u) := \int^{u}_0 \sigma(\tau)\SP d\tau \SP.
\end{equation}

\par\smallskip

\noindent{\bf Relaxation via stress approximation.} Consider the following extended
system which approximates the stress $\sigma(u)$:
\begin{equation}\label{RELAXELASTSYST}
\begin{aligned}
\left(\begin{aligned}
&u\\
&v\\
&\alpha\\
\end{aligned}\right)_t -
\left(\begin{aligned}
&v\\
\alpha&+Eu\\
&0\\
\end{aligned}\right)_x = \frac{1}{\eps}R(u,v,\alpha) + G(u,v,\alpha)
\end{aligned}
\end{equation}
with
\begin{equation*}\label{RGELAST}
\begin{aligned}
R(u,v,\alpha) = \bigl( \, 0, \, 0, \, h(u)-\alpha \, \bigr)^{\T}, \quad G(u,v,\alpha) = \bigl( \, 0, \, g_2(u,v), \,  0 \, \bigr)^{\T} \SP \\[3pt]
\end{aligned}
\end{equation*}
and the function $h(u)$ defined by
\begin{equation}\label{HFUNCDEF}
h(u)= \sigma(u) - Eu \quad  \mbox{with} \quad E>\Gamma.
\end{equation}

\par\smallskip

Observe that as $\eps \to 0$, the variable $\alpha$ tends to its equilibrium state
$\alpha_{eq}=h(u)$. Thus, the corresponding equilibrium states $u_{eq},v_{eq}$ satisfy
\eqref{ELASTSYST}. This motivates the parameterization of the manifold of Maxwellians by
\begin{equation*}\label{MAXWELAST}
M(u,v,\alpha) = \bigl( u, \SP v, \SP h(u) \SP \bigr)^{\T}
\end{equation*}
which implies \eqref{MAXWPAR}. Next, we easily check that
\begin{equation*}
\begin{aligned}
\dim \mathcal{N}(\nabla R (M(u,v))) = 2 , \quad \dim \mathcal{R}(\nabla R (M(u,v))) = 1
\end{aligned}
\end{equation*}
which verifies \eqref{KERGRADR} for $n=2$, $N=3$. Also, the structure of
\eqref{ELASTSYST}, \eqref{RELAXELASTSYST} suggests the choice of the projection matrix
\begin{equation*}
\mathbb{P} =
\begin{bmatrix}
1 \quad  0 \quad 0\\
0 \quad  1 \quad 0\\
\end{bmatrix}
\end{equation*}
for which $\mathbb{P} M(u,v) = (u,v)^\T$,  $\mathbb{P} R(u,v,\alpha)={\bf 0}$ and hence
(H3) is satisfied.

\par\smallskip

At this point, we identify the corresponding entropy-entropy flux pair of the system
\eqref{RELAXELASTSYST} and verify the remaining hypotheses that allow one to apply the
theory developed in the preceding sections. In view of the requirement \eqref{GSTRESS}
for the stress $\sigma(u)$, $h(u):\RR \to \RR$ is strictly decreasing, {\it onto} and
satisfies $h(0)=0$. Hence $h^{-1}:\RR \to \RR$ is well-defined. Then, we set
\begin{equation*}\label{RXENTPAIRELAST}
\begin{aligned}
H(u,v,\alpha) &:= \tfrac{1}{2} v^2 + \tfrac{1}{2}Eu^2 + \alpha u - \int^{\alpha}_0
h^{-1}(\xi)  \SP d\xi
\\\quad Q(u,v,\alpha) &:= -(\alpha + Eu)v.
\end{aligned}
\end{equation*}
It is easy to check that $H,Q$ is the entropy-entropy flux pair for the system
\eqref{RELAXELASTSYST}. Next, we observe that the entropy $H$ maybe written as
\begin{equation}\label{ENTELAST2}
H(u,v,\alpha) = \frac{v^2}{2}  + \frac{\gamma u^2}{4} + \psi(\alpha) + \frac{(\alpha +
\widehat{E}u)^2 }{2\widehat{E}}
\end{equation}
where
\begin{equation*}
\psi(\alpha) :=  \int_0^{\alpha} \biggl( -h^{-1}(\xi) - \frac{\xi}{\widehat{E}} \biggr)
d\xi, \quad \widehat{E}:=E-\frac{\gamma}{2}>0.
\end{equation*}
From \eqref{GSTRESS} we have
\begin{equation}\label{PSIDER}
\begin{aligned}
\psi''(\alpha) &= -{h^{-1}} '(\alpha) - \frac{1}{\widehat{E}}\\
&=\frac{1}{\bigl(E-\sigma'(h^{-1}(\alpha))\bigr)}-\frac{1}{\bigl(E-\frac{\gamma}{2}\bigr)}
\geqslant \frac{\gamma}{2(E-\gamma)(E-\frac{\gamma}{2})}
> 0.
\end{aligned}
\end{equation}
Then \eqref{ENTELAST2}, \eqref{PSIDER} imply that there exist $\mu,\mu'>0$ such that
\begin{equation}\label{ENTUCONVXELAST}
\begin{aligned}
\mu {\bf{I}}  \SP \leqslant \SP \D^2 H(v,u,\alpha) \SP \leqslant \SP
 \mu' {\bf{I}}, \,\quad (u,v,\alpha) \in \RR^3
\end{aligned}
\end{equation}
and hence the pair $H,Q$ satisfies (H4). Next, we compute
\begin{equation*}
\D H(u,v,\alpha) = (Eu + \alpha, \SP v, \SP u-h^{-1}(\alpha))\\
\end{equation*}
and observe that by \eqref{GSTRESS}
\begin{equation}\label{RDISSPELAST}
-\D H (u,v,\alpha) \SP R(u,v,\alpha) = \bigl(u-h^{-1}(\alpha)\bigr)
\bigl(\alpha-h(u)\bigr) \geqslant \frac{1}{E} \bigl(\alpha-h(u)\bigr)^2
\end{equation}
which implies \eqref{RXENTPROP2}.

\par\smallskip

We next check the entropy consistency between the systems \eqref{ELASTSYST},
\eqref{RELAXELASTSYST}. First we observe that
\begin{equation}\label{QRESTRELAST}
q(u,v):=Q(M(u,v)) = (h(u)+Eu)v = \sigma(u)v.
\end{equation}
Also, we have
\begin{equation*}
\begin{aligned}
\eta(u,v):=H(M(u,v)) &= \tfrac{1}{2} v^2 + \Sigma(u) + k(u),
\end{aligned}
\end{equation*}
where
\begin{equation*}
    k(u) := \tfrac{1}{2} Eu^2+ h(u)u- \int_0^{h(u)} \SP
h^{-1}(\xi)\SP d\xi - \int_0^{u} \sigma(\xi)\SP d\xi.
\end{equation*}
From \eqref{HFUNCDEF} it follows that $k(0)=0$, and $k'(u)=0$ for all $u\in\RR$ and hence
\begin{equation}\label{HRESTRELAST}
\eta(u,v)= \tfrac{1}{2} v^2 + \Sigma(u).
\end{equation}
Then, \eqref{ELASTENTPAIR}, \eqref{HRESTRELAST}, \eqref{QRESTRELAST} imply
\eqref{ENTCONSHYP}. Next, notice that
\begin{equation*}\label{MANIFDIST}
|(u, \, v, \, \alpha)^\T - M(u,v)| = |\alpha - h(u)|
\end{equation*}
and hence \eqref{DSPTERMDEF}, \eqref{DSPTERMALT}, and \eqref{RDISSPELAST} imply
\eqref{RLXDSP} with $\nu=\tfrac{1}{E}$.

\par\smallskip

Finally, we observe that
\begin{equation}\label{SRCCONS}
\begin{aligned}
\bigl[ & \D  H(u,v,\alpha) - \D H(M(\bar{u},\bar{v}))\bigr] \bigl[G(u,v,\alpha) -
G(M(\bar{u},\bar{v}))\bigr] \\[2pt]
& =  \bigl(v-\bar{v}\bigr)\bigl(g_2(u,v)-g_2(\bar{u},\bar{v})\bigr)\\[2pt]
& =\bigl[ \D \eta(u,v) - \D \eta(\bar{u},\bar{v})\bigr] \bigl[g(u,v) -
g(\bar{u},\bar{v})\bigr]\\[2pt]
\end{aligned}
\end{equation}
for each $(u,v,\alpha)^{\T}$, $M(\bar{u},\bar{v}) \in \RR^3$. Then, if the source
$g(u,v)$ satisfies \eqref{WKDSPELAST}, then \eqref{SRCCONS} implies \eqref{SRCDSP}. If,
on the other hand, $g(u,v)$ satisfies \eqref{SRCLIPELAST}, then \eqref{SRCCONS} implies
\eqref{SRCLIP}.

\par\smallskip

Thus, if $\bigl\{(u^{\eps},v^{\eps},\alpha^{\eps})\bigr\}$ is  a uniformly bounded family of weak
solutions, one may apply Theorem \ref{locbdd} to establish convergence before formation of shocks. If such a priori information is not available, then, in addition to
\eqref{GSTRESS}-\eqref{SRCLIPELAST}, require that
\begin{equation}\label{DDSGMBNF}
    |\sigma''(u)| \leqslant K, \quad u\in\RR.
\end{equation}
In that case, from \eqref{ENTUCONVXELAST}, \eqref{DDSGMBNF} it follows that
\eqref{EXTSYSBOUNDSHYP}, \eqref{EQSYSBOUNDSHYP} hold and therefore one may apply Theorems
\ref{WKSRCTHM}, \ref{SRCLIPTHM} (depending on the type of source term).

\begin{remark}\rm
Replacing \eqref{WKDSPELAST} with the weakly dissipative condition
$$
    \bigl(g_2(v)-g_2(\bar{v})\bigr)\bigl(v-\bar{v}\bigr) \leqslant - c |v - \bar{v}|^2, \quad \forall \SP
    v,\bar{v}\in \RR,
$$
the relaxation system falls into the framework of \cite{HN03, Yong-2004}, which provides global smooth solutions for small initial data. The case of Lipschitz source terms can also be handled following similar line of argument as in  \cite{HN03, Yong-2004}. Note that the  same follows for the combustion model presented below, which has  a Lipschitz source term.
\end{remark}

\section{Application to Combustion}\label{S7}

The governing equations for chemical reaction from {\it unburnt} gases to {\it burnt}
gases in certain physical regimes read \cite{CHT03}:
\begin{equation}\label{COMBMCONS}
\begin{aligned}
&\del_t v - \del_x u  = 0\\
&\del_t u + \del_x ( P(v,s,Z) ) = 0\\
&\del_t \bigl( E(v,s,Z) + \tfrac{1}{2} u^2 + qZ \bigr)_t + \del_x(u P(v,s,Z)) = r\\
&\del_t Z + K \varphi(\Theta(v,s,Z)) Z = 0.
\end{aligned}
\end{equation}
The state of the gas is characterized by the macroscopic variables: the specific volume
$v(x,t),$  the velocity field $u(x,t),$ the entropy $s(x,t)$ and the mass fraction of the
reactant $Z(x,t),$ whereas the physical properties of the material are reflected through
appropriate constitutive relations which relate the pressure $P(v,s,Z)$, internal energy
$E(v,s,Z)$ with the macroscopic variables. Here, and in what follows, $q$ represents the
difference in the heats between the reactant and the product, $K$ denotes the rate of the
reactant, whereas $\varphi(\theta) \geqslant 0$ is the reaction function. The function
$r(x,t)$ represents a source term (additional radiating heat density).

\par\smallskip

\noindent {\bf Isentropic combustion.} In this section we address the problem of
relaxation to the isentropic combustion model:
\begin{equation}\label{ISCOMB}
\begin{aligned}
\left(\begin{aligned}
&v\\
&u\\
&Z\\
\end{aligned}\right)_t +
\left(\begin{aligned}
-u&\\
P(v,&Z) \\
0&\\
\end{aligned}\right)_x   =
\left(\begin{aligned}
\!0&\\
\!0&\\
\!-K\varphi(&\Theta(v,Z))\\
\end{aligned}\right)
\end{aligned}
\end{equation}
that arises naturally from \eqref{COMBMCONS} by externally regulating $r$ to ensure
$s=s_0$ \cite{Dafermos79}, in which case we suppress variable $s$ and use the notation
\begin{equation}
P(v,Z):=P(v,s_0,Z), \quad \Theta(v,Z):=\Theta(v,s_0,Z).
\end{equation}

\par\smallskip

\noindent{\bf Motivation for assumptions.} Our main objective is to find a proper
extended system associated with the system \eqref{ISCOMB} that models isentropic
processes with specific volume $v$ away from both zero and vacuum, that is, when $v$ has
upper and lower bounds,
\begin{equation}\label{VBND}
v_0 \leqslant v \leqslant V_0 \quad \mbox{for some} \quad v_0,V_0\in(0,\infty).
\end{equation}
For the rest of the paper we assume that the a priori bound \eqref{VBND} holds.

\par\smallskip

The physics of (isentropic) thermodynamical processes determined by the equations
\eqref{COMBMCONS} and compatible with the Clausius-Duhem inequality require the choice of
the pressure $P(v,Z)$ and temperature $\Theta(v,Z)$ which are compatible with the
following properties: for $v\in[v_0, V_0]$, $s=s_0$, and $Z\in[0,1]$
\begin{equation}\label{EQPROP}
\begin{aligned}
 P(v,z) = -\del_v E(v,s_0,Z) >0, \quad \Theta(v,Z) = \del_s E(v,s_0,Z)>0\\
\end{aligned}
\end{equation}
for some (appropriate) energy function
\begin{equation}
E(v,s,Z)>0 \quad \mbox{with} \quad E_Z(v,s_0,Z) > 0.
\end{equation}
We remark that such a function $E(v,s,Z)$ is known to exist for the system
\eqref{COMBMCONS} as long as $v$, $s$ have lower and upper bounds \cite{CHT03}.
\par\smallskip

For technical  convenience, outside of the interval \eqref{VBND}, we redefine the constitutive law $  E(v,s_0,Z)$ ensuring that the functions $P(v,Z)$, $\Theta(v,Z)$ are
defined for all $v\in\RR$, $Z\in[0,1]$ with bounded derivatives as indicated below.

\par\medskip

\noindent{\bf Conditions on $P$, $\Theta$.}
\begin{itemize}
\item[$(a1)$] Motivated by the physical property  $\del_v P <0$ we assume that
\begin{equation}\label{ICPRESSGROWTH}
0 \SP < \SP \gamma \SP < -\del_v P(v,Z)  \SP < \SP \Gamma, \quad v\in\RR,\SP Z\in[0,1].
\end{equation}

\item[$(a2)$] There exists $\bar{C}>0$ \SP such that
\begin{equation}\label{DPZGROWTH}
\quad \Bigl|\int_{0}^{v} P_{ZZ}(\tau,Z)\SP d\tau \Bigr| < \bar{C}, \,\,\quad |\del_Z
P(v,Z)| < \bar{C}, \quad v\in\RR,\SP Z\in[0,1].
\end{equation}

\item[$(a3)$] The composition $\varphi \circ \Theta$ of the rate and constitutive
temperature functions satisfies for some $L>0$
\begin{equation}\label{COMPOSLIP}
\bigl|\varphi(\Theta(v,Z))-\varphi(\Theta(\bar{v},\bar{Z}))\bigr| \, \leqslant \, L
|(v,Z)-(\bar{v},\bar{Z})|
\end{equation}
for all $(v,Z), (\bar{v},\bar{Z})\in \RR\times[0,1].$

\end{itemize}

\par\smallskip

Under $(a1)$-$(a3)$ the system \eqref{ISCOMB} admits an entropy-entropy flux pair
$\bar{\eta},\bar{q}$ of the form:
\begin{equation}\label{ISCOMBENTPAIR}
\begin{aligned}
\bar{\eta}(v,u,Z) &= \frac{1}{2}u^2 - \biggl(\int^{v}_0 \SP P(\tau,Z) \SP d\tau \biggr) +
B(Z)
\\[3pt]
\quad \bar{q}(v,u,Z) &= P(v,Z)u,
\end{aligned}
\end{equation}
where $B(Z)$ is an {\it arbitrary function}.

\par\medskip

\noindent{\bf Relaxation via approximation of pressure.} In the spirit of the example for
the elasticity system \eqref{ELASTSYST} we define
\begin{equation}\label{HDEFIC}
h(v,Z) := -P(v,Z)- Ev \quad \mbox{with} \quad E>\Gamma.
\end{equation}
We now approximate the pressure $P(v,Z)$ by the linear combination $-(\alpha+Ev)$. This
leads to the extended system,
\begin{equation}\label{RELAXCOMBSYST}
\begin{aligned}
\left(\begin{aligned}
&v\\
&u\\
&Z\\
&\alpha\\
\end{aligned}\right)_t -
\left(\begin{aligned}
&u\\
\alpha&+Ev \\
&0\\
&0\\
\end{aligned}\right)_x   =
\frac{1}{\eps}R(v,u,Z,\alpha) + G(v,u,Z,\alpha),
\end{aligned}
\end{equation}
where
\begin{equation} \label{RGCOMB}
\begin{aligned}
R(v,u,Z,\alpha) &= \bigl[ \, 0, \, 0, \, 0, \, h(v,Z)-\alpha \bigr]^{\T}\\
G(v,u,Z,\alpha) &= \bigl[ \, 0, \, 0, \, -K\varphi(\Theta)Z, \,  0  \bigr]^{\T}.
\end{aligned}
\end{equation}

\par\smallskip

Note that as $\eps \to 0$, $\alpha$ tends to its equilibrium state
$\alpha_{eq}=h(v_{eq},Z_{eq})$. Then
\begin{equation*}
    \alpha_{eq}+ E v_{eq} = - P(v_{eq},Z_{eq})
\end{equation*}
and hence
$(v_{eq},u_{eq}, Z_{eq})$ solves \eqref{ISCOMB}. This
motivates the parameterization of the manifold of Maxwellians $\mathcal{M}$ by
\begin{equation*}\label{MAXWCOMB}
M(v,u,Z) = \bigl[ v, \, u, \, Z, \, h(v,Z) \bigr]^{\T}
\end{equation*}
which yields \eqref{MAXWPAR}. Next, we compute
\begin{equation*}\label{DR}
\D R(v,u,Z,\alpha) = \begin{bmatrix}
&0& &0& &0& &0& \\
&0& &0& &0& &0& \\
&0& &0& &0& &0& \\
&h_v(v,Z)&  &0& &h_Z(v,Z)& &-1&  \\
\end{bmatrix}
\end{equation*}
from which we conclude
\begin{equation*}
\begin{aligned}
\dim \mathcal{N} \bigl(\D R (M(v,u,Z)) \bigr) = 3 , \quad \dim \mathcal{R} \bigl( \SP \D
R (M(v,u,Z)) \bigr) = 1
\end{aligned}
\end{equation*}
which verifies \eqref{KERGRADR} for $n=3$, $N=4$. We choose the projection matrix
\begin{equation*}
\mathbb{P} =
\begin{bmatrix}
1 \quad  0 \quad 0 \quad 0 \\
0 \quad  1 \quad 0 \quad 0 \\
0 \quad  0 \quad 1 \quad 0 \\
\end{bmatrix}
\end{equation*}
for which $\mathbb{P} M(v,u,Z) = (v,u,Z)^\T$,  $\mathbb{P} R(v,u,Z,\alpha)={\bf 0}$ and
hence (H3) holds.

\par\medskip

\noindent{\bf Entropy of the extended system.} We next specify the entropy-entropy flux
pair of the relaxation system \eqref{RELAXCOMBSYST}. By \eqref{ICPRESSGROWTH}
\begin{equation}\label{DhvPROP1}
0 \, < \, E-\Gamma  \, <  \,  - h_v(v,Z) \, < \, E-\gamma.
\end{equation}
Hence there exists $j(\alpha,Z): \RR\times [0,1] \to \RR$ such that
\begin{equation}\label{JHPROP}
\begin{aligned}
j (h(v,Z),\!Z ) = v, \quad   h(j(\alpha,Z),\!Z) = \alpha
\end{aligned}
\end{equation}
for all $v,\alpha \in \RR, Z\in[0,1]$. Thus, we define
\begin{equation*}
\begin{aligned}
H(v,u,Z,\alpha) & \SP := \SP  \frac{u^2}{2} - \int^{\alpha}_{h(0,Z)} j(\xi,Z) \SP d\xi  +
\alpha v +
\frac{E v^2}{2}+ B(Z)\\[3pt]
Q(u,v,\alpha) & \SP := \SP -\bigl(\alpha + E v\bigr) u
\end{aligned}
\end{equation*}
with $B(Z)$ an {arbitrary function} such that
\begin{equation}\label{BPROP}
    B''(Z) \SP > \SP m \SP > \SP 0, \quad Z\in [0,1]
\end{equation}
where the constant $m>0$ is to be specified.

\par\medskip

It is easy to check that $H,Q$ is the entropy-entropy flux pair for
\eqref{RELAXCOMBSYST}. To show that $H(U)$ is strictly convex, however,  is less trivial
and therefore, for the convenience of a reader, we provide detailed calculations.
Recalling that $E > \Gamma > \gamma$ we rewrite $H(v,u,Z,\alpha)$ as follows:
\begin{equation}\label{RXENTIC2}
H(u,v,Z,\alpha) =  \biggl( \frac{u^2}{2}+\frac{\gamma v^2}{4} + \psi(\alpha,Z) \biggr) +
\frac{\bigl(\alpha+\widehat{E} v \bigr)^2}{2 \widehat{E}}
\end{equation}
with
\begin{equation*}\label{PSIDEFIC}
\psi(\alpha, Z) \SP := \SP - \int^{\alpha}_{h(0,Z)} \SP j(\xi,Z) \SP d\xi  -
\frac{\alpha^2}{2\widehat{E}} + B(Z), \quad \widehat{E} := E - \frac{\gamma}{2}.
\end{equation*}

\par\medskip

We now show that there exists $\Lambda>0$ such that
\begin{equation}\label{D2PSIEST}
\Lambda^{-1} {\bf I} \, \leqslant \,    \D^2 \psi(\alpha,Z) \, \leqslant \, \Lambda {\bf
I}
\end{equation}
by establishing the bounds on the eigenvalues of
\begin{equation}\label{D2PSI}
\D^2 \psi(\alpha,Z) =\!
\begin{bmatrix}
&\!\!\! -j_{\alpha} (\alpha,Z)- \widehat{E}^{-1}  &\!\!- j_{Z}(\alpha,Z)  \\[6pt]
&\!\!\! -j_Z(\alpha,Z)     & \!\!B''(Z)- \del_{ZZ} \Bigl(\int_{h(0,Z)}^{\alpha} j(\xi,Z)
\SP d\xi \Bigr)
\end{bmatrix} .
\end{equation}
Differentiating \eqref{JHPROP}$_2$ and recalling \eqref{HDEFIC} we get
\begin{equation}\label{DJ}
\begin{aligned}
j_{\alpha}(\alpha,Z) \SP = \SP \frac{1}{ h_v \bigl(j(\alpha,Z), Z \bigr)}, \, \quad
j_{Z}(\alpha,Z) \SP = \frac{P_Z(j(\alpha,Z),Z)}{h_v(j(\alpha,Z),Z)} \SP
\end{aligned}
\end{equation}
and hence by \eqref{DPZGROWTH}, \eqref{DhvPROP1}
\begin{equation}\label{DJEST}
\begin{aligned}
\frac{1}{E-\gamma} \, \leqslant \, -j_{\alpha}(\alpha,Z) \SP \leqslant \SP \frac{1}{E -
\Gamma},  \quad  \bigl|j_{Z}(\alpha,Z)\bigr| \, \leqslant \, \frac{\bar{C}}{E - \Gamma}
\SP.
\end{aligned}
\end{equation}
Then by \eqref{DJEST}$_1$
\begin{equation}\label{D2PSI11EST}
\frac{\gamma}{2(E-\gamma)\widehat{E}} \, \leqslant \, \Bigl[\D^2
\psi(\alpha,Z)\Bigr]_{11} \, \leqslant \,
\frac{\Gamma-\frac{1}{2}\gamma}{(E-\Gamma)\widehat{E}} \, .\\[0pt]
\end{equation}
Next, using \eqref{JHPROP}, \eqref{DJ}$_{1,2}$  we compute
\begin{equation*}\label{JZ}
\begin{aligned}
&\del_{Z} \biggl( \int_{h(0,Z)}^{\alpha} j(\xi,Z)\SP d\xi \biggr) =\\
&\quad =\int_{h(0,Z)}^{\alpha} j_Z(\xi,Z)\SP d\xi - \Bigl( j(h(0,Z),\!Z) h_Z(0,Z) \Bigr) \\
& \quad = \int_{h(0,Z)}^{\alpha} P_Z(j(\xi,Z),Z) j_{\alpha}(\xi,Z) \SP d\xi  =
\int_{0}^{j(\alpha,Z)} P_Z(\tau,Z) \SP d\tau
\end{aligned}
\end{equation*}
and hence
\begin{equation*}
\begin{aligned}
&\del_{ZZ} \biggl( \int_{h(0,Z)}^{\alpha} \SP j(\xi,Z)\SP d\xi \biggr) \\
    &  \qquad = P_Z(j(\alpha,Z),Z) j_Z(\alpha,Z) + \int_0^{j(\alpha,Z)} \!\! P_{ZZ}(\tau,Z)\, d\tau \SP.\\
\end{aligned}
\end{equation*}
Then, by \eqref{DPZGROWTH}, \eqref{DJEST}$_2$ we conclude
\begin{equation}\label{JZZEST}
\biggl| \del_{ZZ} \biggl( \int_{0}^{\alpha} j(\xi,Z)\, d\xi \biggr) \biggr| \, \leqslant
\, \bar{C} \biggl( 1 + \frac{\bar{C}}{E-\Gamma} \biggr).
\end{equation}
The analysis of the above inequalities motivates to choose
\begin{equation*}\label{MMHDEF}
m:= \widehat{m} + {\bar{C}} \biggl( 1 + \frac{{\bar{C}}}{E-\Gamma} \biggr) \quad
\mbox{with}\quad \widehat{m} := \biggl[ \! \Bigl( \frac{{\bar{C}}^2}{E - \Gamma} \Bigr)^2
+ 1 \! \biggr] \frac{2(E-\gamma)\widehat{E}}{\gamma}
\end{equation*}
in which case by \eqref{BPROP}, \eqref{D2PSI} and \eqref{JZZEST} we obtain
\begin{equation}\label{D2PSI22EST}
0 \, < \, \widehat{m} \, \leqslant \, \Bigl[\D^2 \psi(\alpha,Z)\Bigr]_{22} \, \leqslant
\, 2m.
\end{equation}
Combining \eqref{DJEST}$_2$, \eqref{D2PSI11EST}, and \eqref{D2PSI22EST} we get
\begin{equation}\label{abc}
1 \, \leqslant \, \det \Bigl[\D^2 \psi(\alpha,Z) \Bigr] = \lambda_1 \lambda_2,
\end{equation}
where $\lambda_1, \lambda_2 \in \RR$ denote the largest and smallest eigenvalues of $\D^2
\psi$, respectively. Observe that \eqref{D2PSI}, \eqref{D2PSI11EST}, and
\eqref{D2PSI22EST} imply
\begin{equation}\label{LMBDEF}
 0 < \, \lambda_1 \, \leqslant \,  \Lambda := \biggl(2m +
\frac{\Gamma-\frac{\gamma}{2}}{(E-\Gamma)\widehat{E}} + \frac{{\bar{C}}^2}{E - \Gamma}
\biggr) \SP.
\end{equation}
Then, from \eqref{abc}, \eqref{LMBDEF} we obtain the estimate \eqref{D2PSIEST}.


\par\smallskip

Combining \eqref{RXENTIC2}, \eqref{D2PSIEST}  we conclude that for some $\mu,\mu'>0$
\begin{equation}\label{D2RXPENTEST}
\begin{aligned}
\mu \SP {\bf{I}}  \, \leqslant \, \D^2 H(v,u,Z,\alpha) \, \leqslant  \, \mu' \SP {\bf{I}}
\end{aligned}
\end{equation}
and this yields \eqref{RXENTPROP1}.

\par\smallskip

Now, recalling \eqref{RGCOMB}, \eqref{JHPROP}, and \eqref{DJEST}$_1$  we obtain
\begin{equation*}\label{RDISSPIC}
\begin{aligned}
-&\D H (v,u,Z,\alpha) \SP R(v,u,Z,\alpha)\\[1pt]
&\qquad = \,-\bigl(j(h(v,Z),Z)-j(\alpha,Z)\bigr) \bigl(h(v,Z)-\alpha\bigr)\\[0pt]
&\qquad = \, -\Bigl[\int_0^1 \SP j_{\alpha}\bigl(s h(v,Z) + (1-s)\alpha,Z \bigr)\SP ds
\Bigr]   \bigl(h(v,Z) - \alpha\bigr)^2 \\[0pt]
 &\qquad \geqslant \SP \frac{1}{E-\gamma} \SP
\bigl(h(v,Z)-\alpha\bigr)^2 \SP = \SP \frac{1}{E-\gamma}\SP \bigl|M(v,u,Z) -
(v,u,Z,\alpha)^{\T} \bigr|^2
\end{aligned}
\end{equation*}
which implies that the entropy $H$ satisfies hypotheses \eqref{RXENTPROP2},
\eqref{RLXDSP}.

\par\medskip

Next, we observe that \eqref{HDEFIC}, \eqref{JHPROP}, and \eqref{DJ}$_1$ imply
\begin{equation*}\label{INTJ}
\begin{aligned}
\int^{h(v,Z)}_{h(0,Z)} \SP j(\xi,Z)\SP d\xi &
= \int^{v}_{0} \SP h_v(\tau,Z)\tau \SP d\tau 
= h(v,Z)v + \frac{E v^2}{2} + \int^{v}_{0}  P(v,Z) \SP d\tau.
\end{aligned}
\end{equation*}
Thus, the entropy pair $H,Q$ restricted to the equilibrium manifold satisfies
\begin{equation}\label{HQRESTR}
\begin{aligned}
 \eta(v,u,Z) &:= H(M(v,u,Z))
 = \frac{u^2}{2} - \int^{v}_{0} P(v,Z) \SP d\tau + B(Z)\\
 q(v,u,Z)&:=Q(M(v,u,Z)) = - \bigl(h(v,Z)+Ev \bigr)u = P(v,Z)u.
\end{aligned}
\end{equation}
Then, \eqref{HQRESTR} together with \eqref{ISCOMBENTPAIR} yields \eqref{ENTCONSHYP}.

\par\medskip

Consider an arbitrary compact set $\mathcal{A}\subset \RR\times\RR\times[0,1]\times\RR$.
Then, by \eqref{COMPOSLIP} for all
\begin{equation*}
\begin{aligned}
(v,u,Z,\alpha) \in \RR\times\RR\times[0,1]\times\RR, \quad
(\bar{v},\bar{u},\bar{Z},\bar{\alpha}) \in \mathcal{A},
\end{aligned}
\end{equation*}
we have
\begin{equation}\label{GLIPCOMBEST}
\begin{aligned}
& \bigl|G(v,u,Z,\alpha) - G(\bar{v},\bar{u},\bar{Z},\bar{\alpha}) \bigr| \\[2pt]
& \;\quad = \, \bigl|K \varphi(\Theta(v,Z))Z - K \varphi(\Theta(\bar{v},\bar{Z}))\bar{Z}\bigr|\\[0pt]
& \;\quad \leqslant \, |K| \Bigl(\SP \bigl|Z\bigr|\bigl| \SP \varphi(\Theta(v,Z)) -
  \varphi(\Theta(\bar{v},\bar{Z}))\bigr| + \bigl|\varphi(\Theta(\bar{v},\bar{Z}))||Z-\bar{Z}|
  \Bigr)\\[0pt]
  & \;\quad  \leqslant \, (L+L_{\mathcal{A}})  |K|
  \bigl|(v,u,Z,\alpha)-(\bar{v},\bar{u},\bar{Z},\bar{\alpha})\bigr|,
\end{aligned}
\end{equation}
where  $L_{\mathcal{A}}>0$ denotes a constant for which, in view of \eqref{COMPOSLIP},
there holds
\begin{equation*}\label{PHIBNDONCOMP}
\bigl|\varphi(\Theta(\bar{v},\bar{Z}))\bigr| \leqslant L_{\mathcal{A}}, \quad
(\bar{v},\bar{u},\bar{Z},\bar{\alpha}) \in \mathcal{A}.
\end{equation*}

\par\smallskip

 The estimate \eqref{GLIPCOMBEST} implies that the source $G$ satisfies the hypothesis \eqref{SRCLIP} on
$\RR\times\RR\times[0,1]\times\RR$, the state space of  \eqref{RELAXCOMBSYST} with initial data such that $0 \leqslant
Z(\cdot,0) \leqslant 1 $. Thus, if the family
$\bigl\{(v^{\eps},u^{\eps},Z^{\eps},\alpha^{\eps})\bigr\}$ is uniformly bounded, one may
apply Theorem \ref{locbdd} to establish convergence before the formation of shocks. If such a
priori information is not available, then, in addition to $(a1)$-$(a3)$, require that
\begin{equation}\label{DDPB}
    |\D^2 P(v,Z) | \leqslant K, \quad |B'''(Z)|< K, \quad v\in\RR, Z\in[0,1]\in\RR.
\end{equation}
In that case, from \eqref{D2RXPENTEST}, \eqref{DDPB} it follows that
\eqref{EXTSYSBOUNDSHYP}, \eqref{EQSYSBOUNDSHYP} hold and therefore one may apply Theorem
\ref{SRCLIPTHM}.

\section{General framework for symmetric hyperbolic systems, $d=1$}\label{S8}

In this section we present a general strategy indicating how starting from a symmetric
hyperbolic system one can construct an extended relaxation system.

\par\medskip

Consider the hyperbolic balance law
\begin{equation}\label{SYMSYST}
    \del_t u + \del_x f(u) = g(u), \quad u,f(u),g(u)\in \RR^n
\end{equation}
such that:
\begin{itemize}
    \item The flux $f(u)$ has {\it symmetric} $\D_{u}f(u)$. Thus,
\begin{equation}\tag{h1}\label{SYMFLUXPROP}
f(u) = \D \Phi^{\T}(u) \quad \mbox{for some} \quad \Phi(u):\RR^n \to \RR.
\end{equation}

\item $\Phi$ is {\it convex} and for some $\Gamma,\gamma>0$ such that
\begin{equation}\tag{h2}\label{PHIBOUNDSYM}
0<\gamma < \D^2 \Phi(u) < \Gamma, \quad  \quad  u\in \RR^n.
\end{equation}

\item For each compact $\mathcal{A} \subset \RR^n$ there exists $L_{\mathcal{A}}>0$ such
that
\begin{equation}\tag{h3}\label{SRCSYM}
\bigl|g(u)-g(\bar{u})\bigr| \SP \leqslant \SP L_{\mathcal{A}} \SP |u-\bar{u}|,  \,\quad
\SP u\in \RR^n, \,\bar{u}\in \mathcal{A}.
\end{equation}
\end{itemize}
By \eqref{SYMFLUXPROP} the system \eqref{SYMSYST} admits entropy-entropy flux pair
\begin{equation*}
\bar{\eta}(u) = \Phi(u), \quad \, \bar{q}(u) = \frac{1}{2}|\D \Phi(u)|^2.
\end{equation*}

\par\smallskip

\noindent{\bf Relaxation via flux approximation.} Next, we approximate the flux $f(u)$ by
the combination $\alpha + {\D \mathcal{E}}^{\T}(u)$, where $\alpha\in\RR^n$ is a new
vector variable and $\mathcal{E}(u):\RR^n \to \RR$ is a convex function such that for
some $E,\delta>0$  there holds
\begin{equation}\tag{h4}\label{EPROPSYM}
\begin{aligned}
(E+\delta) \SP {\bf I}  \SP \geqslant \SP \D^2_{u} \mathcal{E}(u) \SP \geqslant \SP E \SP
{\bf I}, \,\quad E >\Gamma > \gamma > \delta >0.
\end{aligned}
\end{equation}
This leads to the relaxation system for variables $u,\alpha\in\RR^n$
\begin{equation}\label{RELAXSYMSYST}
\begin{aligned}
\left(\begin{aligned}
&u\\
&\alpha\\
\end{aligned}\right)_t +
\left(\begin{aligned}
\alpha&+{\D\mathcal{E}}^{\T}(u)\\
&0\\
\end{aligned}\right)_x   =
\frac{1}{\eps}R(u,\alpha) + G(u,\alpha)
\end{aligned}
\end{equation}
with
\begin{equation} \label{RGSYM}
\begin{aligned}
R(u,\alpha) = \bigl[\SP 0, \, h(u)-\alpha \SP \bigr]^{\T},\quad G(u,\alpha) = \bigl[\SP
g(u),\, 0 \SP \bigr]^{\T}.
\end{aligned}
\end{equation}

\par\smallskip

We now define
\begin{equation}\label{SGMHDEF}
\begin{aligned}
\Sigma(u) := \mathcal{E}(u)-\Phi(u), \quad h(u):= -\D_{u} \Sigma^{\T}(u) = f(u)-\D_{u}
{\mathcal{E}}^{\T}(u).
\end{aligned}
\end{equation}
Then, by \eqref{PHIBOUNDSYM}, \eqref{EPROPSYM} we have
\begin{equation}\label{SGMCONVX}
\begin{aligned}
\D^2_{u} \mathcal{E} \SP \geqslant \SP E\SP {\bf I} \SP > \SP (E +\delta -\gamma)\SP{\bf
I} \SP > \SP \D^2_{u} \Sigma \SP
> \SP (E-\Gamma)\SP{\bf I}.
\end{aligned}
\end{equation}
The mapping
\begin{equation}\label{DSGMONTO}
{\D_{u} \Sigma: \RR^n \to \RR^n \quad \mbox{is onto}}
\end{equation}
as implied by the following lemma (c.f. Zeidler \cite{Zeidler}).

\begin{lemma}
Suppose $V(u):\RR^n \to \RR^n$ is a $C^1$-mapping such that
\begin{equation*}
\D V(u):\RR^n \to \RR^n \,\, \mbox{is invertible for all $u\in\RR^n$}
\end{equation*}
and the map $V(u)$ is coercive, that is,  $V(u)^{\T} {u}> c |u|^2$, $u\in\RR^n$ for some
fixed constant $c>0$.
Then, $V$ must be {\it surjective} and covers all of $\RR^n$.
\end{lemma}

\par\smallskip

Observe that as $\eps \to 0$,  $\alpha$ tends to its equilibrium state
$\alpha_{eq}=h(u_{eq})$ in which case the corresponding equilibrium state $u_{eq}$
satisfies \eqref{SYMSYST}. This suggests the parameterization  of the manifold of
Maxwellians by
\begin{equation*}
M(u) = \bigl[ \SP u, \SP h(u) \SP \bigr]^{\T}
\end{equation*}
which yields \eqref{MAXWPAR}. Next, observe that
\begin{equation*}
\begin{aligned}
\dim \mathcal{N}(\D R (M(u))) = n , \quad \dim \mathcal{R}(\D R (M(u))) = n
\end{aligned}
\end{equation*}
which verifies \eqref{KERGRADR} with $N=2n$. The structure of \eqref{SYMSYST},
\eqref{RELAXSYMSYST} suggests the choice of the projection matrix
\begin{equation*}
\mathbb{P} = \bigl[ \,{\bf{I}},  {\bf{0}} \, \bigr] :\RR^{2n} \to \RR^n \quad \mbox{for
which} \quad \mathbb{P} M(u) = u,\,\,  \mathbb{P} R(u,\alpha)={\bf 0}
\end{equation*}
which implies (H3).

\par\smallskip

To construct the entropy-entropy flux pair for the relaxation system \eqref{RELAXSYMSYST}
we exploit the ideas of the analysis of A. Tzavaras \cite{Tz06}. By \eqref{SGMCONVX},
\eqref{DSGMONTO} the map $\D \Sigma$ is bijective. This motivates the definition of
\begin{equation}\label{jDEF}
j(\alpha):\RR^n \to \RR^n \quad \mbox{by} \quad    j(\alpha)^{\T} = - (\D_{u}
\Sigma)^{-1}(-\alpha), \quad \alpha\in\RR^n.
\end{equation}
Then, by the inverse mapping theorem, $\D_{\alpha} j(\alpha)$ is {\it symmetric} and
hence there exists $J(\alpha):\RR^n \to \RR\SP$ such that
\begin{equation}\label{JPROP}
\begin{aligned}
 \D_{\alpha}J(\alpha) &= j(\alpha)^{\T}=- (\D \Sigma)^{-1}(-\alpha)\\[3pt]
 \D^2_{\alpha} J(\alpha) &= \bigl[ \D^2_{u} \Sigma(-\D_{\alpha}
 J^{\T}(\alpha))\bigr]^{-1}=\bigl[ \D^2_{u} \Sigma(-j(\alpha))\bigr]^{-1}.\\[3pt]
\end{aligned}
\end{equation}
Furthermore, by \eqref{SGMCONVX} we obtain that $J(\alpha)$ is uniformly convex with
\begin{equation}\label{JCONVX}
(E+\Gamma)^{-1} \SP {\bf I} \SP > \SP \D^2_{\alpha} J(\alpha)  \SP > \SP
(E+\delta-\gamma)^{-1} \SP {\bf I}.
\end{equation}
We next define
\begin{equation}\label{RXENTPAIRSYM}
\begin{aligned}
H(u,\alpha) &= \mathcal{E}(u) + \alpha^{\T}u + J(\alpha)\\[2pt]
Q(u,\alpha) &= \frac{1}{2} \bigl| \SP \alpha + {\D \mathcal{E}}^{\T}(u)\bigr|^2.
\end{aligned}
\end{equation}
It is easy to verify that $H,Q$ is the entropy-entropy flux pair for the system
\eqref{RELAXSYMSYST}. To show that $H(u,\alpha)$ is strictly convex, we compute the
Hessian
\begin{equation*}\label{D2HSYM}
\D^2 H(u,\alpha) =
\begin{bmatrix}
\!\!\!\!&\D^2_{u}\mathcal{E}(u)   &{\bf I} \\[3pt]
\!\!\!\!&{\bf I}               &\D^2_{\alpha}J(\alpha)\\
\end{bmatrix}
\end{equation*}
and write
\begin{equation*}
\begin{aligned}
(u,\alpha)^{\T}  \bigl[\D^2 H(u,\alpha)\bigr](u,\alpha)= \, u^{\T} \bigl[\D^2
\mathcal{E}(u)\bigr]u + 2\SP\alpha^{\T}u + \alpha^{\T}\bigl[\D^2 J(\alpha)\bigr]\alpha.
\end{aligned}
\end{equation*}
Then, recalling \eqref{EPROPSYM}, \eqref{SGMCONVX}, and \eqref{JCONVX} we get the
estimates
\begin{equation*}
\begin{aligned}
(u,\alpha)^{\T} \bigl[\D^2 H(u,\alpha)\bigr](u,\alpha)
 > \SP \tfrac{1}{2}(\gamma-\delta)|u|^2 +
\frac{\tfrac{1}{2}(\gamma-\delta)|\alpha|^2}{\bigl(E+\tfrac{1}{2}(\delta-\gamma)\bigr)\bigl(E+\delta-\gamma\bigl)}
\end{aligned}
\end{equation*}
and
\begin{equation*}
\begin{aligned}
(u,\alpha)^{\T}  \bigl[\D^2 H(u,\alpha)\bigr](u,\alpha) \SP \leqslant \SP (E+\delta+1)
|u|^2 + \bigl( (E-\Gamma)^{-1}+1 \bigr) |\alpha|^2.
\end{aligned}
\end{equation*}
The above inequalities and the fact that $\gamma>\delta$ imply that there exist
$\mu,\mu'>0$ such that
\begin{equation}\label{ENTUCONVSYM}
\mu' {\bf I} \SP \leqslant \SP \D^2 H(u,\alpha) \SP \leqslant \mu \SP {\bf I}, \quad
(u,\alpha)\in\RR^{n+n}
\end{equation}
and hence we conclude that the pair $H,Q$ satisfies \eqref{RXENTPROP1}.

\par\smallskip

Next, we compute
\begin{equation*}
\D H(u,\alpha) = \bigl[\SP \D_{u} \mathcal{E}(u)+\alpha^{\T}, \, u^{\T}+
\D_{\alpha}J(\alpha) \SP \bigr]
\end{equation*}
and observe that by \eqref{SGMHDEF}$_1$, \eqref{jDEF},
\begin{equation}\label{jhcomp}
-j(h(u)) = (\D \Sigma)^{-1}(-h(u))  =  (\D \Sigma)^{-1}(\D \Sigma^{\T}(u)) = u.
\end{equation}
Hence recalling \eqref{RGSYM}$_1$, \eqref{SGMHDEF},  \eqref{JPROP}$_1$, and
\eqref{JCONVX} we obtain
\begin{equation*}
\begin{aligned}
-\D H & (u,\alpha) R(u,\alpha)= \\[2pt]
&=\bigl(u^{\T} + \D_{\alpha}J(\alpha)
\bigr)\bigl(\alpha- h(u) \bigr)\\[2pt]
& = \bigl( j(\alpha)-j(h(u)) \bigr)^{\T}  \bigl(\alpha - h(u) \SP
\bigr)\\[2pt]
& = (\alpha-h(u)\bigr)^{\T} \biggr[ \int^{1}_0 \, \D^2 J \bigl(s\alpha +
(1-s)h(u)\bigr)\,ds \biggl]\bigl(\alpha-h(u)\bigr)\\
& \geqslant \frac{1}{(E+\delta-\gamma)}\bigl|\alpha-h(u)\bigr|^2 =
\frac{1}{(E+\delta-\gamma)}\bigl|(u,\alpha)^{\T}-M(u)\bigr|^2.
\end{aligned}
\end{equation*}
The last inequality implies that $H$ satisfies hypotheses \eqref{RXENTPROP2},
\eqref{RLXDSP}.

\par\smallskip

Next, observe that by \eqref{SGMHDEF}, \eqref{RXENTPAIRSYM}
\begin{equation}\label{HONMANIFSYM}
\begin{aligned}
\eta(u):=H(M(u)) &= H(u,h(u))\\
&= \mathcal{E}(u) + h(u)^{\T}u  + J(h(u))\\
&= \Phi(u)+\Sigma(u) - [\D_{u}\Sigma(u)]u + J(h(u)).\\
\end{aligned}
\end{equation}
Then, by \eqref{SGMHDEF}, \eqref{JPROP}$_1$, and \eqref{jhcomp}
\begin{equation*}
\begin{aligned}
\D_u \eta(u) &= \D_{u} \Phi(u)-\bigl[\D^2_{u}\Sigma(u)\bigr]u +
j(h(u))^{\T}\bigl[-\D^2_{u}\Sigma(u)\bigr]
 = \D_{u} \Phi(u)
\end{aligned}
\end{equation*}
and we conclude
\begin{equation*}\label{ETAFORM}
    \eta(u) = \Phi(u) + C \quad \mbox{for some} \quad C\in\RR.
\end{equation*}
Similarly, by \eqref{SGMHDEF}, \eqref{RXENTPAIRSYM}
\begin{equation}\label{QONMANIFSYM}
\begin{aligned}
q(u):=Q(M(u)) = Q(u,h(u))= \tfrac{1}{2} \bigl|h(u) + \D \mathcal{E}^{\T}(u) \bigr|^2
=\tfrac{1}{2} \bigl|{\D\Phi}(u) \bigr|^2.
\end{aligned}
\end{equation}
By the discussion in the beginning of the section we conclude that $\eta,q$ defined in
\eqref{HONMANIFSYM}, \eqref{QONMANIFSYM} is an entropy-entropy flux pair of
\eqref{SYMSYST} which implies \eqref{ENTCONSHYP}.

\par\medskip


Now, take an arbitrary {\it compact} set $\mathcal{C}\subset \RR^n\times\RR^n$ and
define
\begin{equation*}
\mathcal{A} = \bigl\{\bar{u}\in\RR^n: \, \mbox{for some $\bar{\alpha}\in\RR^n$}
\,\,(\bar{u},\bar{\alpha})\in \mathcal{C} \bigr\}
\end{equation*}
which is compact as well. Then, by \eqref{SRCSYM} for all $(u,\alpha) \in \RR^{n+n}$,
$(\bar{u},\bar{\alpha}) \in \mathcal{C}$
\begin{equation*}
    |G(u,\alpha)-G(\bar{u},\bar{\alpha})|=|g(u)-g(\bar{u})|
    \leqslant L_{\mathcal{A}}|(u,\alpha)-(\bar{u},\bar{\alpha})|.
\end{equation*}
The above estimate shows that $G(u,\alpha)$ satisfies \eqref{SRCLIP}. Thus, the
relaxation system \eqref{RELAXSYMSYST} satisfies (H1)-(H7), \eqref{SRCLIP}. Thus, if
$\bigl\{u^{\eps}\bigr\}$ is uniformly  bounded family of weak solutions, one may apply
Theorem \ref{locbdd} to establish convergence. If such a priori information is not
available, then, in addition to \eqref{SYMFLUXPROP}-\eqref{EPROPSYM}, require that
\begin{equation}\label{D3PHI}
    |\D^3 \Phi(u) | \leqslant K, \quad u\in\RR^n.
\end{equation}
In that case, from \eqref{ENTUCONVSYM}, \eqref{D3PHI} it follows that
\eqref{EXTSYSBOUNDSHYP}, \eqref{EQSYSBOUNDSHYP} hold and therefore one may apply Theorem
\ref{SRCLIPTHM} to establish convergence in the smooth regime.

\medskip

\section{Acknowledgements}
 The authors thank Thanos Tzavaras for several fruitful conversations and suggestions during
the course of this investigation. A.M. thanks Robin Young for advice and helpful comments
regarding the combustion model. K.T. acknowledges the support by the National Science
foundation under the grant DMS-1211519 and the support by the Simons Foundation under the
grant $\# 267399$.


\begin{thebibliography}{99}
\setlength{\parskip}{1em}


\bibitem{Brenier2000}
{Y.\ Brenier}, {\em Convergence of the Vlasov-Poisson system to the incompressible Euler
equations}, Comm.\ Partial Diff.\ Equations, {\bf 25}, 737-754, 2000.

\bibitem{BNP2004} Y.\ Brenier, R.\ Natalini and M.\ Puel, {\em On the relaxation approximation  of the incompressible Navier-Stokes equations}, Proc.\ Amer.\ Math.\ Soc., {\bf 132}, 1021-1028, 2004.

\bibitem{BV2005} F.\ Berthelin and A.\ Vasseur,
{\em From Kinetic Equations to Multidimensional Isentropic Gas Dynamics Before Shocks},
{SIAM J. Math. Anal.,} {\bf 36,} 1807-1835, 2005.

\bibitem{BV2009} F.\ Berthelin, A.\ Vasseur and A.\ Tzavaras, {\em From discrete velocity Boltzmann equations to gas dynamics before shocks} {J.
Stat. Physics,} {\bf 135} (2009), 151-173

\bibitem{Caflish79} R.\ Caflish and G.\ Papanicolaou,
{\em The fluid-dynamical limit of a nonlinear model Boltzmann equation}, {Comm. Pure Appl. Appl.,} {\bf 32}, 589-616, 1979.

\bibitem{CHENLEVLIU} G.Q. Chen, C.D. Levermore and T.P. Liu,  {\em Hyperbolic conservation laws with stiff relaxation.},
{Lett. Math. Phys.,} {\bf 22 (1)}, (1991), 63-80.

\bibitem{CHT02} G.-Q. Chen, D. Hoff and K. Trivisa, {\em On the Navier-Stokes Equations for Exothermically, Reacting, Compressible Fluids.} {Acta Math. Appl. Sinica,} 18, (2002), 15-36.

\bibitem{CHT03} G.-Q. Chen, D. Hoff and K. Trivisa, {\em Global Solutions to a Model for Exothermically Reacting, Compressible Flows with Large Discontinuous Initial Data.} {Arch. Ration. Mech. Anal.,} 166, (2003), 321-358.


\bibitem{Dafermos79}
{C.M. Dafermos}, {\em The second law of thermodynamics and stability.} {Arch. Rational
Mech. Anal.,} {\bf 70} (1979), 167-179.

\bibitem{Dafermos86}
{C.M. Dafermos}, {\em Quasilinear hyperbolic systems with involutions}, {Arch. Rational
Mech. Anal.,} {\bf 94}, 373--389 (1986).

\bibitem{Dafermos06}
{C.M. Dafermos}, {\em Hyperbolic systems of balance laws with weak dissipation}, {Journal
of Hyperbolic Differential Equations,} {\bf Vol. 3}, {\bf No. 3}, 505-527 (2006).

\bibitem{Dafermos10}
{C.M. Dafermos},  {\em Hyperbolic conservation laws in continuum physics.} Third edition.
Grundlehren der Mathematischen Wissenschaften, 325. Springer-Verlag, Berlin, 2010).

\bibitem{DiPerna79}
{R. DiPerna}, {\em Uniqueness of solutions to hyperbolic conservation laws}, {Indiana U.
Math. J.,} {\bf 28}, 137--188 (1979).

\bibitem{GLT-2009} P.\ Goncalves, C.\ Landim, and C. \ Toninelli. {\em Hydrodynamic limit for a particle system with degenerate rates.} {Ann. Inst. H. Poincar� Probab. Statist.,} {\bf 45}, {\bf 4}, 887-909, 2009.

\bibitem{JK-2000} S.\ Jin, M.\  Katsoulakis,
{\em Hyperbolic Systems with Supercharacteristic Relaxations and Roll Waves.}  {SIAM J.
Appl. Math.,}  {\bf 61} no. 1, 273�292 (2000).

\bibitem{HN03} B.\ Hanouzet and R.\ Natalini, {\em Global existence of smooth solutions for partially
dissipative hyperbolic systems with a convex entropy.} {Arch. Rational Mech. Anal.,} {\bf
169} (2003), 89�117.


\bibitem{KMT-2012} T.\ Karper, A.\ Mellet and K.\ Trivisa, {\em Hydrodynamic limit of the kinetic Cucker-Smale flocking model.,} submitted (2012).

\bibitem{LT}
{C. Lattanzio, A.E. Tzavaras}, {\em Structural properties of stress relaxation and
convergence from viscoelasticity to polyconvex elastodynamics}, {Arch. Rational Mech.
Anal.,} {\bf 180}, 449--492 (2006).


\bibitem{MT10}
A. Miroshnikov, A. Tzavaras, {\em A Variational approximation scheme for radial polyconvex elasticity that preserves the positivity of Jacobians}, Comm. Math. Sci., {\bf
10}-1, 87--115 (2012).



\bibitem{MT11}
{A. Miroshnikov, A. Tzavaras, {\em Convergence of variational approximation schemes for elastodynamics with polyconvex energy},  Zeit. Anal. Anwend. (to appear).}


\bibitem{Mourragui-1996} M. \ Mourragui. {\em Comportement hydrodynamique et entropie relative des processus de sauts, de naissances et de morts.} {Ann. Inst. H. Poincar\'{e} Probab. Statist.,} {\bf 32} 361�385, 1996.

\bibitem{MV2008} A.\ Mellet and A.\ Vasseur. {\em Asymptotic Analysis for a Vlasov-Fokker-Planck Compressible Navier-Stokes Systems of Equations.} {Commun. Math. Phys.,} {\bf 281}, 573-596, (2008).

\bibitem{Tz05}
{A. Tzavaras}, {\em Relative Entropy in Hyperbolic Relaxation.} Comm. Math. Sci. (2005).

\bibitem{Tz06}
{A. Tzavaras}, {\em A relaxation theory with polyconvex entropy function converging to elastodynamics}, preprint 2012.


\bibitem{Yau-1991} H.T. Yau, { \em Relative entropy and hydrodynamics of Ginzburg-Landau models}, {Lett. Math. Phys.,} {\bf 22 (1)} (1991) 63�80.

\bibitem{Yong-2004} W.-A. Yong, {\em Entropy and global existence for hyperbolic balance laws.} {Arch. Rational Mech. Anal.,} {\bf 172} (2004) 247-266.

\bibitem{Zeidler} E. Zeidler, {\em Nonlinear Functional Analysis and its applications, Part I.} {Springer.,} (1986).

\end{thebibliography}
\end{document}